\documentclass[12pt]{article}
\usepackage{amsmath,amssymb,amsthm,amscd,graphicx}

\theoremstyle{plain}
\newtheorem{theorem}{Theorem}[section]
\newtheorem{lemma}[theorem]{Lemma}
\newtheorem{conjecture}[theorem]{Conjecture}
\theoremstyle{definition}
\newtheorem{definition}[theorem]{Definition}
\theoremstyle{remark}
\newtheorem{remark}[theorem]{Remark}

\numberwithin{equation}{section}

\begin{document}

\title{The Hilbert--Smith conjecture for three-manifolds}

\author{John Pardon}

\date{9 April 2012; Revised 22 January 2013}

\maketitle

\begin{abstract}
We show that every locally compact group which acts faithfully on a connected three-manifold is a Lie group.  By known reductions, it suffices to show that there is no faithful action of $\mathbb Z_p$ (the $p$-adic integers) on a connected three-manifold.  If $\mathbb Z_p$ acts faithfully on $M^3$, we find an interesting $\mathbb Z_p$-invariant open set $U\subseteq M$ with $H_2(U)=\mathbb Z$ and analyze the incompressible surfaces in $U$ representing a generator of $H_2(U)$.  It turns out that there must be one such incompressible surface, say $F$, whose isotopy class is fixed by $\mathbb Z_p$.  An analysis of the resulting homomorphism $\mathbb Z_p\to\operatorname{MCG}(F)$ gives the desired contradiction.  The approach is local on $M$.

MSC 2010 Primary: 57S10, 57M60, 20F34, 57S05, 57N10

MSC 2010 Secondary: 54H15, 55M35, 57S17

\end{abstract}

\section{Introduction}

By a \emph{faithful action} of a topological group $G$ on a topological manifold $M$, we mean a continuous injection $G\to\operatorname{Homeo}(M)$ (where $\operatorname{Homeo}(M)$ has the compact-open topology).  A well-known problem is to characterize the locally compact topological groups which can act faithfully on a manifold.  In particular, there is the following conjecture, which is a natural generalization of Hilbert's Fifth Problem.

\begin{conjecture}[Hilbert--Smith]\label{hilbert}
If a locally compact topological group $G$ acts faithfully on some connected $n$-manifold $M$, then $G$ is a Lie group.
\end{conjecture}

It is well-known (see, for example, Lee \cite{lee} or Tao \cite{taoblog}) that as a consequence of the solution of Hilbert's Fifth Problem (by Gleason \cite{gleason,gleason2}, Montgomery--Zippin \cite{montgomeryzippin,montgomeryzippin2}, as well as further work by Yamabe \cite{yamabe1,yamabe2}), any counterexample $G$ to Conjecture \ref{hilbert} must contain an embedded copy of $\mathbb Z_p$ (the $p$-adic integers).  Thus it is equivalent to consider the following conjecture.

\begin{conjecture}\label{hs}
There is no faithful action of $\mathbb Z_p$ on any connected $n$-manifold $M$.
\end{conjecture}

Conjecture \ref{hilbert} also admits the following reformulation, which we hope will help the reader better understand its flavor.

\begin{conjecture}\label{homeonscs}
Given a connected $n$-manifold $M$ with metric $d$ and open set $U\subseteq M$, there exists $\epsilon>0$ such that the subset:
\begin{equation}
\left\{\phi\in\operatorname{Homeo}(M)\bigm|d(x,\phi(x))<\epsilon\text{ for all }x\in U\right\}
\end{equation}
of $\operatorname{Homeo}(M)$ contains no nontrivial compact subgroup.
\end{conjecture}

For any specific manifold $M$, Conjectures \ref{hilbert}--\ref{homeonscs} for $M$ are equivalent.  They also have the following simple consequence for \emph{almost periodic} homeomorphisms of $M$ (a homeomorphism is said to be almost periodic if and only if the subgroup of $\operatorname{Homeo}(M)$ it generates has compact closure; see Gottschalk \cite{gottschalk} for other equivalent definitions).

\begin{conjecture}\label{alp}
For every almost periodic homeomorphism $f$ of a connected $n$-manifold $M$, there exists $r>0$ such that $f^r$ is in the image of some homomorphism $\mathbb R\to\operatorname{Homeo}(M)$.
\end{conjecture}

Conjecture \ref{hilbert} is known in the cases $n=1,2$ (see Montgomery--Zippin \cite[pp233,249]{montgomeryzippin2}).  By consideration of $M\times\mathbb R$, clearly Conjecture \ref{hilbert} in dimension $n$ implies the same in all lower dimensions.

The most popular approach to the conjectures above is via Conjecture \ref{hs}.  Yang \cite{ctyang} showed that for any counterexample to Conjecture \ref{hs}, the orbit space $M/\mathbb Z_p$ must have cohomological dimension $n+2$.  Conjecture \ref{hs} has been established for various regularity classes of actions, $C^2$ actions by Bochner--Montgomery \cite{bochnermontgomery}, $C^{0,1}$ actions by Repov{\u{s}}--{\u{S}}{\u{c}}epin \cite{hslipschitz}, $C^{0,\frac n{n+2}+\epsilon}$ actions by Maleshich \cite{hsholder}, quasiconformal actions by Martin \cite{hsqc}, and uniformly quasisymmetric actions on doubling Ahlfors regular compact metric measure manifolds with Hausdorff dimension in $[1,n+2)$ by Mj \cite{mj}.  In the negative direction, it is known by work of Walsh \cite[p282 Corollary 5.15.1]{walsh} that there \emph{does} exist a continuous decomposition of any compact PL $n$-manifold into cantor sets of arbitrarily small diameter if $n\geq 3$ (see also Wilson \cite[Theorem 3]{wilson}).  By work of Raymond--Williams \cite{raymondwilliams}, there are faithful actions of $\mathbb Z_p$ on $n$-dimensional compact metric spaces which achieve the cohomological dimension jump of Yang \cite{ctyang} for every $n\geq 2$.

In this paper, we establish the aforementioned conjectures for $n=3$.

\begin{theorem}\label{hs3}
There is no faithful action of $\mathbb Z_p$ on any connected three-manifold $M$.
\end{theorem}

\subsection{Rough outline of the proof of Theorem \ref{hs3}}\label{shortproof}

We suppose the existence of a continuous injection $\mathbb Z_p\to\operatorname{Homeo}(M)$ and derive a contradiction.

Since $p^k\mathbb Z_p\cong\mathbb Z_p$, we may replace $\mathbb Z_p$ with one of its subgroups $p^k\mathbb Z_p$ for any large $k\geq 0$.  The subgroups $p^k\mathbb Z_p\subseteq\mathbb Z_p$ form a neighborhood base of the identity in $\mathbb Z_p$; hence by continuity of the action, as $k\to\infty$ these subgroups converge to the identity map on $M$.  By picking a suitable Euclidean chart of $M$ and a suitably large $k\geq 0$, we reduce to the case where $M$ is an open subset of $\mathbb R^3$ and the action of $\mathbb Z_p$ is very close to the identity.\footnote{There are two motivations for this reduction (which is valid in any dimension).  First, recall that a topological group is NSS (``has no small subgroups'') iff there exists an open neighborhood of the identity which contains no nontrivial subgroup.  Then a theorem of Yamabe \cite[p364 Theorem 3]{yamabe2} says that a locally compact topological group is a Lie group iff it is NSS.  Thus the relevant property of $\mathbb Z_p$ which distinguishes it from a Lie group is the existence of the small subgroups $p^k\mathbb Z_p$.  Second, recall Newman's theorem \cite{newman} which implies that a compact Lie group acting nontrivially on a manifold must have large orbits.  In essence, we are extending Newman's theorem to the group $\mathbb Z_p$ (however the proof will be quite different).}

The next step is to produce a compact connected $\mathbb Z_p$-invariant subset $Z\subseteq M$ satisfying the following two properties:\footnote{The motivation to consider such a set is to attempt a \emph{dimension reduction} argument.  In other words, we would like to conclude that $\partial Z$ is a closed surface with a faithful action of $\mathbb Z_p$, and therefore contradicts the (known) case of Conjecture \ref{hs} with $n=2$.  This, of course, is not possible since $Z$ could \emph{a priori} have wild boundary.  We will nevertheless be able to construct a closed surface $F$ which serves as an ``approximate boundary'' of $Z$.}
\begin{enumerate}
\item On a coarse scale, $Z$ looks like a handlebody of genus two.
\item The action of $\mathbb Z_p$ on $H^1(Z)$ is nontrivial.
\end{enumerate}
The eventual contradiction will come by combining the first (coarse) property of $Z$ with the second (fine) property of $Z$.  Constructing such a set $Z$ follows a natural strategy: we take the orbit of a closed handlebody of genus two and attach the orbit of a small loop connecting two points on the boundary.  However, the construction requires checking certain connectedness properties of a number of different orbit sets, and is currently the least transparent part of the proof.

Now we consider an open set $U$ defined roughly as $N_\epsilon(Z)\setminus Z$ (only roughly, since we need to ensure that $U$ is $\mathbb Z_p$-invariant and $H_2(U)=\mathbb Z$).  The set of isotopy classes of incompressible surfaces in $U$ representing a generator of $H_2(U)$ forms a lattice, and this lets us find an incompressible surface $F$ in $U$ which is fixed up to isotopy by $\mathbb Z_p$.  We think of the surface $F$ as a sort of ``approximate boundary'' of $Z$.  Even though $\mathbb Z_p$ does not act naturally on $F$ itself, we do get a natural homomorphism $\mathbb Z_p\to\operatorname{MCG}(F)$ with finite image.  The two properties of $Z$ translate into the following two properties of the action of $\mathbb Z_p$ on $H_1(F)$:
\begin{enumerate}
\item There is a submodule of $H_1(F)$ fixed by $\mathbb Z_p$ on which the intersection form is:
\begin{equation}\label{rankfour}
\left(\begin{matrix}0&1\cr-1&0\end{matrix}\right)\oplus\left(\begin{matrix}0&1\cr-1&0\end{matrix}\right)
\end{equation}
\item The action of $\mathbb Z_p$ on $H_1(F)$ is nontrivial.
\end{enumerate}
This means we have a cyclic subgroup $\mathbb Z/p\subseteq\operatorname{MCG}(F)$ such that $H_1(F)^{\mathbb Z/p}$ has a submodule on which the intersection form is given by (\ref{rankfour}).  The Nielsen classification of cyclic subgroups of the mapping class group shows that this is impossible.  This contradiction completes the proof of Theorem \ref{hs3}.

We conclude this introduction with a few additional remarks on the proof.  First, the proof is a local argument on $M$, similar in that respect to the proof of Newman's theorem \cite{newman}.  Second, our proof works essentially verbatim with any pro-finite group in place of $\mathbb Z_p$ (though this is not particularly surprising).  Finally, we remark that assuming the action of $\mathbb Z_p$ on $M$ is \emph{free} (as is traditional in some approaches to Conjecture \ref{hs}) does not produce any significant simplifications to the argument.

\subsection{Acknowledgements}

We thank Ian Agol for suggesting Lemmas \ref{posetcomp} and \ref{lattice} concerning lattice properties of incompressible surfaces.  We also thank Mike Freedman and Steve Kerckhoff for helpful conversations.  We thank the referee for their comments and for giving this paper a very close reading.

The author was partially supported by a National Science Foundation Graduate Research Fellowship under grant number DGE--1147470.

\section{The lattice of incompressible surfaces}\label{incompressiblesection}

In this section, we study a natural partial order on the set of incompressible surfaces in a particularly nice class of three-manifolds; we call such three-manifolds ``quasicylinders'' (see Definition \ref{approxsurf}).  For a quasicylinder $M$, we let $\mathcal S(M)$ denote the set of isotopy classes of incompressible surfaces in $M$ which generate $H_2(M)$.  The main result of this section (suggested by Ian Agol \cite{mathoverflow}) is that $\mathcal S(M)$ is a lattice (see Lemma \ref{lattice}) under its natural partial order.  Related ideas may be found in papers of Schultens \cite{schultens} and Przytycki--Schultens \cite{przytyckischultens} on the contractibility of the Kakimizu complex \cite{kakimizu}.\footnote{The Kakimizu complex is a sort of ``$\mathbb Z$-equivariant order complex'' of $\mathcal S(\widetilde{\mathbb S^3\setminus K})$, where $\widetilde{\mathbb S^3\setminus K}$ denotes the infinite cyclic cover of the knot complement $\mathbb S^3\setminus K$.  Even though technically $\widetilde{\mathbb S^3\setminus K}$ is not a quasicylinder under our definition, it is easy to give a modified definition (allowing manifolds with boundary) to which the methods of this section would apply.}

Since we will ultimately use the results of this section to study groups of \emph{homeomorphisms}, we need constructions which are functorial with respect to homeomorphisms.  To study the properties of these constructions, however, we use methods and results in PL/DIFF three-manifold theory (for example, general position).  Thus in this section we will always state explicitly which category (TOP/PL/DIFF) we are working in, since this will change frequently (and there is no straightforward way of working just in a single category).  We will, of course, need the key result that every topological three-manifold can be triangulated, as proved by Moise \cite{moise} and Bing \cite{bing} (see also Hamilton \cite{hamilton} for a modern proof based on the methods of Kirby--Siebenmann \cite{kirsie}; the PL structure is unique, but we do not need this).

By surface, we always mean a closed connected orientable surface.  By isotopy (resp.\ PL isotopy) of surfaces, we always mean ambient isotopy through homeomorphisms (resp.\ PL homeomorphisms).

\begin{theorem}[{\cite[p62 Theorem 8]{bing}}]\label{bingplapproxthm}
Let $M_1,M_2$ be two PL three-manifolds where $M_2$ has a metric $d$.  Let $\phi:M_1\to M_2$ be a homeomorphism, and let $f:M_1\to\mathbb R_{>0}$ be continuous.  Then there exists a PL homeomorphism $\phi_1:M_1\to M_2$ with $d(\phi(x),\phi_1(x))\leq f(x)$.
\end{theorem}

\begin{lemma}\label{straightening}
Let $M$ be a PL three-manifold, and let $F$ be a bicollared surface in $M$.  Then there exists an isotopy of $M$ supported in an arbitrarily small neighborhood of $F$ which maps $F$ to a PL surface.
\end{lemma}

\begin{proof}
Let $\phi:F\times[-1,1]\to M$ be a bicollar, which we may assume is arbitrarily close to $F=\phi(F\times\{0\})$.  Now pick a PL structure on $F$ and apply Theorem \ref{bingplapproxthm} to $\phi|_{F\times(0,1)}$ and a function $f$ which decreases sufficiently rapidly near the ends of $(0,1)$.  The resulting $\phi_1$ then extends continuously to $F\times[0,1]$ and agrees with $\phi$ on $F\times\{0,1\}$.  Now splice $\phi_1|_{F\times[0,1]}$ and $\phi|_{F\times[-1,0]}$ together to get a bicollar $\phi_2:F\times[-1,1]\to M$ which is PL on $F\times(0,1)$.  Now using the bicollar $\phi_2$ we can easily construct an isotopy of $M$ which sends $F=\phi_2(F\times\{0\})$ to $\phi_2(F\times\{\frac 12\})$, which is PL.
\end{proof}

\begin{lemma}\label{homotopicisotopic}
Let $M$ be an irreducible orientable TOP (resp.\ PL) three-manifold, and let $F_1,F_2$ be two bicollared (resp.\ PL) $\pi_1$-injective surfaces in $M$.  If $F_1,F_2$ are homotopic, then there is a compactly supported (resp.\ PL) isotopy of $M$ which sends $F_1$ to $F_2$.
\end{lemma}

\begin{proof}
Waldhausen \cite[p76 Corollary 5.5]{waldhausen} proves this in the PL category if $M$ is compact with boundary.  It is clear that this implies our PL statement, since the given homotopy will be supported in a compact region of $M$, which is in turn contained in a compact irreducible submanifold with boundary.

For the TOP category, first pick a PL structure on $M$, and use Lemma \ref{straightening} to straighten $F_1,F_2$ by a compactly supported isotopy.  Now use the PL version of this lemma.
\end{proof}

\begin{definition}\label{approxsurf}
A \emph{quasicylinder} is an irreducible orientable three-manifold $M$ with exactly two ends and $H_2(M)\cong\mathbb Z$.  For example, $\Sigma_g\times\mathbb R$ is a quasicylinder for $g\geq 1$.
\end{definition}

\begin{lemma}
Let $M$ be a quasicylinder.  For an embedded surface $F\subseteq M$, the following are equivalent:\\
\indent 1.\ $F$ is nonzero in $H_2(M)$.\\
\indent 2.\ $F$ separates the two ends of $M$.\\
\indent 3.\ $F$ generates $H_2(M)$.
\end{lemma}

\begin{proof}
(1)$\implies$(2).  A path from one end to the other gives a non-torsion class in $H_1^{\operatorname{lf}}(M)$.  Thus its Poincar\'e dual is a non-torsion class in $H^2(M)$, and thus defines a nonzero map $H_2(M)\to\mathbb Z$.  Since $F$ is nonzero in $H_2(M)\cong\mathbb Z$, every such path must therefore intersect $F$.

(2)$\implies$(3).  If $F$ separates the two ends, then there is a path from one end to the other which intersects $F$ exactly once.  Thus the Poincar\'e dual of the class of this path in $H_1^{\operatorname{lf}}(M)$ evaluates to $1$ on $F\in H_2(M)$.  Thus $F$ represents a primitive element of $H_2(M)\cong\mathbb Z$, and thus generates it.

(3)$\implies$(1).  Trivial.
\end{proof}

\begin{definition}
For a TOP quasicylinder $M$, let $\mathcal S_{\operatorname{TOP}}(M)$ be the set of bicollared $\pi_1$-injective embedded surfaces in $M$ generating $H_2(M)$, modulo homotopy.
\end{definition}

\begin{definition}
For a PL quasicylinder $M$, let $\mathcal S_{\operatorname{PL}}(M)$ be the set of PL $\pi_1$-injective embedded surfaces in $M$ generating $H_2(M)$, modulo homotopy.
\end{definition}

\begin{remark}\label{nonempty}
$\mathcal S_{\operatorname{PL}}(M)$ is always nonempty, since we can pick a PL embedded surface representing a generator of $H_2(M)$ and then take some maximal compression thereof (as in the proof of Lemma \ref{maximalcompression} below), which will be $\pi_1$-injective by the loop theorem.
\end{remark}

\begin{definition}
A \emph{directed quasicylinder} is a quasicylinder along with a labelling of the ends with $\pm$.  For an embedded surface $F\subseteq M$ separating the two ends, let $(M\setminus F)_\pm$ denote the two connected components of $M\setminus F$ (the subscripts corresponding to the labelling of the ends).  It is easy to see that $(M\setminus F)_\pm$ are both (directed) quasicylinders.
\end{definition}

\begin{definition}
Let $M$ be a directed quasicylinder.  For two embedded surfaces $F_1,F_2\subseteq M$ separating the two ends of $M$, we say $F_1\leq F_2$ iff $F_1\subseteq(M\setminus F_2)_-$.  Define a relation $\leq$ on $\mathcal S_{\operatorname{TOP}}(M)$ (resp.\ $\mathcal S_{\operatorname{PL}}(M)$) by declaring that $\mathfrak F_1\leq\mathfrak F_2$ if and only if $\mathfrak F_1,\mathfrak F_2$ have embedded representatives $F_1,F_2$ with $F_1\leq F_2$.
\end{definition}

\begin{lemma}\label{pltopequiv}
Let $M$ be a PL directed quasicylinder.  Then the natural map $\psi:\mathcal S_{\operatorname{PL}}(M)\to\mathcal S_{\operatorname{TOP}}(M)$ is a bijection satisfying $\mathfrak F_1\leq\mathfrak F_2\iff\psi(\mathfrak F_1)\leq\psi(\mathfrak F_2)$.
\end{lemma}

\begin{proof}
The natural map $(\mathcal S_{\operatorname{PL}}(M),\leq)\to(\mathcal S_{\operatorname{TOP}}(M),\leq)$ is clearly injective, and by Lemma \ref{straightening} it is surjective.

If $\mathfrak F_1,\mathfrak F_2$ have bicollared representatives $F_1,F_2$ with $F_1\leq F_2$, then by Lemma \ref{straightening} they can be straightened preserving the relation $F_1\leq F_2$.  The other direction is obvious, since PL surfaces have a bicollar.
\end{proof}

Having established that $\mathcal S_{\operatorname{PL}}(M)$ and $\mathcal S_{\operatorname{TOP}}(M)$ are naturally isomorphic, we henceforth use the notation $\mathcal S(M)$ for both.

\begin{lemma}\label{arcconst}
Let $M$ be a PL quasicylinder.  Suppose $F$ is a PL embedded surface in $M$ separating the two ends of $M$, and suppose $\gamma$ is a PL arc from one end of $M$ to the other which intersects $F$ transversally in exactly one point.  Denote by $\pi_1(M,\gamma)$ one of the groups $\{\pi_1(M,p)\}_{p\in\gamma}$ (they are all \emph{naturally} isomorphic given the path $\gamma$).\footnote{A categorical way of phrasing this is as follows.  Let $\boldsymbol\gamma$ denote the category whose objects are points $p\in\gamma$, with a single morphism $p\to p'$ for all $p,p'\in\gamma$.  The fundamental group is a functor $\pi_1(M,\cdot):\boldsymbol\gamma\to\mathfrak G\mathfrak r\mathfrak o\mathfrak u\mathfrak p\mathfrak s$, where the morphism $p\to p'$ is sent to the isomorphism $\pi_1(M,p)\to\pi_1(M,p')$ given by the path from $p$ to $p'$ along $\gamma$.  Then $\pi_1(M,\gamma)$ is defined as the limit/colimit of this functor.}  Then for any surface $G\subseteq M$ homotopic to $F$, there is a canonical map $\pi_1(G)\to\pi_1(M,\gamma)$ (defined up to inner automorphism of the domain).
\end{lemma}

\begin{proof}
Assume that $G$ intersects $\gamma$ transversally.  Let us call two intersections of $\gamma$ with $G$ \emph{equivalent} iff there is a path between them on $G$ which, when spliced with the path between them on $\gamma$, becomes null-homotopic in $M$ (this is an equivalence relation).  Note that during a general position homotopy of $G$, the $\operatorname{mod}2$ cardinalities of the equivalence classes of intersections with $\gamma$ remain the same.  Thus since $G$ is homotopic to $F$ and $\#(F\cap\gamma)=1$, there is a unique equivalence class of intersection $G\cap\gamma$ of odd cardinality.  Picking any one of these points as basepoint on $G$ and on $\gamma$ gives the same map $\pi_1(G)\to\pi_1(M,\gamma)$ up to inner automorphism of the domain.  This map is clearly constant under homotopy of $G$.
\end{proof}

\begin{lemma}\label{uniquemcg}
Let $M$ be a PL quasicylinder.  Let $F$ be a PL $\pi_1$-injective surface in $M$ separating the two ends of $M$.  Then any homotopy of $F$ to itself induces the trivial element of $\operatorname{MCG}(F)$.
\end{lemma}

\begin{proof}
Pick a PL arc $\gamma$ from one end of $M$ to the other which intersects $F$ transversally exactly once.  By Lemma \ref{arcconst}, we get a canonical map $\pi_1(F)\to\pi_1(M,\gamma)$ which is constant as we move $F$ by homotopy.  Since this map is injective, we know $\pi_1(F)$ up to inner automorphism as a subgroup of $\pi_1(M,\gamma)$.
\end{proof}

\begin{lemma}\label{posetcomp}
Let $M$ be a directed quasicylinder.  Then the pair $(\mathcal S(M),\leq)$ is a partially-ordered set.  That is, for all $\mathfrak F_1,\mathfrak F_2,\mathfrak F_3\in\mathcal S(M)$ we have:\\
\indent 1.\ (reflexivity) $\mathfrak F_1\leq\mathfrak F_1$.\\
\indent 2.\ (antisymmetry) $\mathfrak F_1\leq\mathfrak F_2\leq\mathfrak F_1\implies\mathfrak F_1=\mathfrak F_2$.\\
\indent 3.\ (transitivity) $\mathfrak F_1\leq\mathfrak F_2\leq\mathfrak F_3\implies\mathfrak F_1\leq\mathfrak F_3$.
\end{lemma}

\begin{proof}
By Lemma \ref{pltopequiv}, it suffices to work in the PL category.

Reflexivity is obvious.

For transitivity, suppose $\mathfrak F_1\leq\mathfrak F_2\leq\mathfrak F_3$.  Then pick representatives $F_1,F_2,F_2',F_3'$ such that $F_1\leq F_2$ and $F_2'\leq F_3'$.  By Lemma \ref{homotopicisotopic}, there is a PL isotopy from $F_2$ to $F_2'$.  Applying this isotopy to $F_1$ produces $F_1'\leq F_2'\leq F_3'$.

For antisymmetry, suppose $\mathfrak F_1\leq\mathfrak F_2$ and $\mathfrak F_2\leq\mathfrak F_1$.  Pick PL representatives $F_1\leq F_2\leq F_1'\leq F_2'$ (this is possible by the argument used for transitivity).  Now pick a PL arc $\gamma$ from one end of $M$ to the other which intersects each of $F_1,F_1',F_2,F_2'$ exactly once.  Since the maps $\pi_1(F_1),\pi_1(F_1'),\pi_1(F_2),\pi_1(F_2')\to\pi_1(M,\gamma)$ are injective, we will identify each of the groups on the left with its image in $\pi_1(M,\gamma)$.  By Lemma \ref{arcconst}, we have $\pi_1(F_i)=\pi_1(F_i')$ under this identification.

By van Kampen's theorem we have:
\begin{equation}\label{amalg}
\pi_1(M,\gamma)=\pi_1((M\setminus F_2)_-,\gamma)\mathop\ast_{\pi_1(F_2)}\pi_1((M\setminus F_2)_+,\gamma)
\end{equation}
Since $F_1\leq F_2\leq F_1'$, we have $\pi_1(F_1)\subseteq\pi_1((M\setminus F_2)_-,\gamma)$ and $\pi_1(F_1')\subseteq\pi_1((M\setminus F_2)_+,\gamma)$.  Since $\pi_1(F_1)=\pi_1(F_1')$, we have $\pi_1(F_1)\subseteq\pi_1((M\setminus F_2)_-,\gamma)\cap\pi_1((M\setminus F_2)_+,\gamma)$.  It follows from the properties of the amalgamated product (\ref{amalg}) that this intersection is just $\pi_1(F_2)$.  Thus we have $\pi_1(F_1)\subseteq\pi_1(F_2)$.  A symmetric argument shows the reverse inclusion, so we have $\pi_1(F_1)=\pi_1(F_2)$.  Now an easy obstruction theory argument shows $F_1$ and $F_2$ are homotopic (using the fact that $\pi_2(M)=0$ by the sphere theorem \cite{papa}).  Thus $\mathfrak F_1=\mathfrak F_2$.
\end{proof}

\begin{lemma}\label{subposetone}
Let $M$ be a directed quasicylinder.  Let $F$ be a bicollared embedded $\pi_1$-injective surface.  Then the natural map $\psi:\mathcal S((M\setminus F)_-)\to\mathcal S(M)$ is a bijection of $\mathcal S((M\setminus F)_-)$ with the set $\left\{\mathfrak G\in\mathcal S(M)\bigm|\mathfrak G\leq[F]\right\}$.  Furthermore, this bijection satisfies $\mathfrak G_1\leq\mathfrak G_2\iff\psi(\mathfrak G_1)\leq\psi(\mathfrak G_2)$.
\end{lemma}

\begin{proof}
By Lemma \ref{pltopequiv}, it suffices to work in the PL category.

Certainly the image of $\psi$ is contained in $\left\{\mathfrak G\in\mathcal S(M)\bigm|\mathfrak G\leq[F]\right\}$.  If $\mathfrak G\leq[F]$, then there are representatives $G'\leq F'$.  By Lemma \ref{homotopicisotopic}, there is a PL isotopy sending $F'$ to $F$.  Applying this isotopy to $G'$ gives a representative $G\leq F$.  Thus the image of $\psi$ is exactly $\left\{\mathfrak G\in\mathcal S(M)\bigm|\mathfrak G\leq[F]\right\}$.  To prove that $\psi$ is injective, suppose $G_1,G_2\leq F$ are two $\pi_1$-injective embedded surfaces which are homotopic in $M$.  Pick an arc $\gamma$ from one end of $M$ to the other which intersects $G_1,F$ exactly once.  Since $G_1,G_2$ are homotopic, Lemma \ref{arcconst} gives canonical maps $\pi_1(G_1),\pi_1(G_2)\to\pi_1(M,\gamma)$ with the same image.  Now these both factor through $\pi_1((M\setminus F)_-,\gamma)\to\pi_1(M,\gamma)$, which is injective since $\pi_1(F)\to\pi_1((M\setminus F)_\pm,\gamma)$ are injective.  Thus $\pi_1(G_1),\pi_1(G_2)\to\pi_1((M\setminus F)_-,\gamma)$ have the same image, so the same obstruction theory argument used in the proof of Lemma \ref{posetcomp} implies that $G_1,G_2$ are homotopic in $(M\setminus F)_-$.  Thus $\psi$ is injective.

Now it remains to show that $\psi$ preserves $\leq$.  The only nontrivial direction is to show that $\psi(\mathfrak G_1)\leq\psi(\mathfrak G_2)\implies\mathfrak G_1\leq\mathfrak G_2$.  If $\psi(\mathfrak G_1)\leq\psi(\mathfrak G_2)$, then we have representatives $G_1'\leq G_2'\leq F'$ (by the argument for transitivity in Lemma \ref{posetcomp}).  Now by Lemma \ref{homotopicisotopic} there is a PL isotopy from $F'$ to $F$, and applying this isotopy to $G_1',G_2'$, we get $G_1\leq G_2\leq F$.  Since $\psi$ is injective, we have $[G_i]=\mathfrak G_i$ in $\mathcal S((M\setminus F)_-)$.  Thus $\mathfrak G_1\leq\mathfrak G_2$.
\end{proof}

\begin{lemma}\label{subposettwo}
Let $M$ be a directed quasicylinder.  Let $F_1\leq F_2$ be two bicollared embedded $\pi_1$-injective surfaces.  Then the natural map $\psi:\mathcal S((M\setminus F_1)_+\cap(M\setminus F_2)_-)\to\mathcal S(M)$ is a bijection of $\mathcal S((M\setminus F_1)_+\cap(M\setminus F_2)_-)$ with $\left\{\mathfrak G\in\mathcal S(M)\bigm|[F_1]\leq\mathfrak G\leq[F_2]\right\}$.  Furthermore, this bijection satisfies $\mathfrak G_1\leq\mathfrak G_2\iff\psi(\mathfrak G_1)\leq\psi(\mathfrak G_2)$.
\end{lemma}

\begin{proof}
This is just two applications of Lemma \ref{subposetone}.
\end{proof}

\begin{lemma}\label{maximalcompression}
Let $M$ be a PL directed quasicylinder.  Let $G\subseteq M$ be any PL embedded surface (not necessarily $\pi_1$-injective) separating the two ends of $M$.  Then there exists a class $\mathfrak G\in\mathcal S(M)$ such that for all $\mathfrak F\in\mathcal S(M)$ with a representative $F\leq G$ (resp.\ $F\geq G$), we have $\mathfrak F\leq\mathfrak G$ (resp.\ $\mathfrak F\geq\mathfrak G$).
\end{lemma}

\begin{proof}
As long as $G$ is compressible, we can perform the following operation.  Do some compression on $G$ (this may disconnect $G$), and pick one of the resulting connected components which is nonzero in $H_2(M)$ to keep.  This operation does not destroy the property that $G$ can be isotoped to lie in $(M\setminus F)_+$ (resp.\ $(M\setminus F)_-$) for an incompressible surface $F$.  Since each step decreases the genus of $G$, we eventually reach an incompressible surface, which by the loop theorem is $\pi_1$-injective, and thus defines a class $\mathfrak G\in\mathcal S(M)$.
\end{proof}

\begin{lemma}\label{existupperandlowerbounds}
Let $M$ be a directed quasicylinder.  Let $A\subseteq\mathcal S(M)$ be a finite set.  Then there exist elements $\mathfrak A_-,\mathfrak A_+\in\mathcal S(M)$ such that $\mathfrak A_-\leq\mathfrak A\leq\mathfrak A_+$ for all $\mathfrak A\in A$.
\end{lemma}

\begin{proof}
By Lemma \ref{pltopequiv}, it suffices to work in the PL category.

Pick PL transverse representatives $A_1,\ldots,A_n$ of all the surfaces in $A$.  Let $((M\setminus A_1)_\pm\cap\cdots\cap(M\setminus A_n)_\pm)_0$ denote the unique unbounded component of $(M\setminus A_1)_\pm\cap\cdots\cap(M\setminus A_n)_\pm$, and define $A_\pm=\partial(((M\setminus A_1)_\pm\cap\cdots\cap(M\setminus A_n)_\pm)_0)$.  It is easy to see that $A_\pm$ are connected and separate the two ends of $M$.  Now apply Lemma \ref{maximalcompression} to $A_-$ and $A_+$ to get the desired classes $\mathfrak A_-,\mathfrak A_+\in\mathcal S(M)$.
\end{proof}

\begin{lemma}[suggested by Ian Agol \cite{mathoverflow}]\label{lattice}
Let $M$ be a directed quasicylinder.  Then the partially-ordered set $(\mathcal S(M),\leq)$ is a lattice.  That is, for $\mathfrak F_1,\mathfrak F_2\in\mathcal S(M)$, the following set has a least element:
\begin{equation}
X(\mathfrak F_1,\mathfrak F_2)=\left\{\mathfrak H\in\mathcal S(M)\bigm|\mathfrak F_1,\mathfrak F_2\leq\mathfrak H\right\}
\end{equation}
(and the same holds in the reverse ordering).
\end{lemma}

\begin{proof}
By Lemma \ref{pltopequiv}, it suffices to work in the PL category.

First, let us deal with the case where $M$ has tame ends, that is $M$ is the interior of a compact PL manifold with boundary $(M,\partial M)$.  Equip $(M,\partial M)$ with a smooth structure and a smooth Riemannian metric so that the boundary is convex.

Now let us recall some facts about area minimizing surfaces.  The standard existence theory of Schoen--Yau \cite{schoenyau} and Sacks--Uhlenbeck \cite{sacksuhlenbeck,sacksuhlenbeck2} gives the existence of area minimizing maps in any homotopy class of $\pi_1$-injective surfaces.  By Osserman \cite{osserman1,osserman2} and Gulliver \cite{gulliver}, such area minimizing maps are immersions.  Now Freedman--Hass--Scott \cite{freedman} have proved theorems to the effect that area minimizing representatives of $\pi_1$-injective surfaces in irreducible three-manifolds have as few intersections as possible given their homotopy class.  Though they state their main results only for closed irreducible three-manifolds, they remark \cite[p634 \S 7]{freedman} that their results remain true for compact irreducible three-manifolds with convex boundary.

Suppose $\mathfrak F_1,\mathfrak F_2\in\mathcal S(M)$; let us show that $X(\mathfrak F_1,\mathfrak F_2)$ has a least element.  Pick area minimizing representatives $F_1,F_2$ in the homotopy classes of $\mathfrak F_1,\mathfrak F_2$.  By \cite[p626 Theorem 5.1]{freedman}, $F_1,F_2$ are embedded.  Let $((M\setminus F_1)_+\cap(M\setminus F_2)_+)_0$ denote the unique unbounded component of $(M\setminus F_1)_+\cap(M\setminus F_2)_+$, and define $F=\partial(((M\setminus F_1)_+\cap(M\setminus F_2)_+)_0)$.  It is easy to see that $F$ is connected and separates the two ends of $M$.  Note that $F_1,F_2$ can be isotoped to lie in $(M\setminus F)_-$.  Furthermore, if $\mathfrak G\in X(\mathfrak F_1,\mathfrak F_2)$ (and $\mathfrak G\ne\mathfrak F_1,\mathfrak F_2$), then $\mathfrak G$ has a least area representative $G$, which by \cite[p630 Theorem 6.2]{freedman} is disjoint from $F_1$ and $F_2$.  Thus $F_1,F_2\leq G$, so $F\leq G$.  Thus we may apply Lemma \ref{maximalcompression} to $F$ and produce an element $\mathfrak F\in\mathcal S(M)$ which is a least element of $X(\mathfrak F_1,\mathfrak F_2)$.  This finishes the proof for $M$ with tame ends.

Now let us treat the case of general $M$.  Consider the following sets:
\begin{equation}
X(\mathfrak F_1,\mathfrak F_2;\mathfrak G)=\left\{\mathfrak H\in\mathcal S(M)\bigm|\mathfrak F_1,\mathfrak F_2\leq\mathfrak H\leq\mathfrak G\right\}\subseteq X(\mathfrak F_1,\mathfrak F_2)
\end{equation}
We claim that it suffices to show that if $\mathfrak G\in X(\mathfrak F_1,\mathfrak F_2)$, then $X(\mathfrak F_1,\mathfrak F_2;\mathfrak G)$ has a least element.  To see this, we argue as follows.  If $\mathfrak G_1,\mathfrak G_2\in X(\mathfrak F_1,\mathfrak F_2)$ and $\mathfrak G_1\leq\mathfrak G_2$, then it is easy to see that the natural inclusion $X(\mathfrak F_1,\mathfrak F_2;\mathfrak G_1)\to X(\mathfrak F_1,\mathfrak F_2;\mathfrak G_2)$ sends the least element of the domain to the least element of the target (assuming both sets have least elements).  Since for every $\mathfrak G_1,\mathfrak G_2$, there exists $\mathfrak G_3$ greater than both (by Lemma \ref{existupperandlowerbounds}), we see that the least elements of all the sets $\{X(\mathfrak F_1,\mathfrak F_2;\mathfrak G)\}_{\mathfrak G\in X(\mathfrak F_1,\mathfrak F_2)}$ are actually the same element $\mathfrak F\in X(\mathfrak F_1,\mathfrak F_2)$.  We claim that this $\mathfrak F$ is a least element of $X(\mathfrak F_1,\mathfrak F_2)$.  This is trivial: if $\mathfrak G\in X(\mathfrak F_1,\mathfrak F_2)$, then $\mathfrak F$ is a least element of $X(\mathfrak F_1,\mathfrak F_2;\mathfrak G)$, so \emph{a fortiori} $\mathfrak F\leq\mathfrak G$.  Thus it suffices to show that each set $X(\mathfrak F_1,\mathfrak F_2;\mathfrak G)$ has a least element.

Let us show that $X(\mathfrak F_1,\mathfrak F_2;\mathfrak G)$ has a least element.  By Lemma \ref{existupperandlowerbounds} there exists $\mathfrak F_0\in\mathcal S(M)$ so that $\mathfrak F_0\leq\mathfrak F_1,\mathfrak F_2$.  Now pick PL representatives $F_0,G$ of $\mathfrak F_0,\mathfrak G$ with $F_0\leq G$.  Let $M_0=(M\setminus F_0)_+\cap(M\setminus G)_-$, which is a directed quasicylinder with tame ends.  Thus by the case dealt with earlier, $\mathcal S(M_0)$ is a lattice, so $\mathfrak F_1,\mathfrak F_2$ have a least upper bound in $\mathcal S(M_0)$.  By Lemma \ref{subposettwo}, $\mathcal S(M_0)\to\mathcal S(M)$ is an isomorphism onto the subset $\left\{\mathfrak F\in\mathcal S(M)\bigm|\mathfrak F_0\leq\mathfrak F\leq\mathfrak G\right\}$.  Thus we get the desired least element of $X(\mathfrak F_1,\mathfrak F_2;\mathfrak G)$.
\end{proof}

\begin{remark}
It should be possible to prove an existence result for area minimizing surfaces in \emph{any} DIFF quasicylinder (with appropriate conditions on the Riemannian metric near the ends).  In that case, the argument used for tame quasicylinders would apply in general (the results of Freedman--Hass--Scott \cite{freedman} extend as long as one has the existence of area minimizing representatives of $\pi_1$-injective surfaces in $M$).
\end{remark}

\begin{remark}
Jaco--Rubinstein \cite{jacorubinstein} have developed a theory of \emph{normal surfaces} in triangulated three-manifolds analogous to that of minimal surfaces in smooth three-manifolds with a Riemannian metric.  In particular, they have results analogous to those of Freedman--Hass--Scott \cite{freedman}.  Thus it is likely that we could eliminate entirely the use of the smooth category and the results in minimal surface theory for the proof of Lemma \ref{lattice}.
\end{remark}

\section{Tools applicable to arbitrary open/closed subsets of manifolds}\label{toolssec}

In our study of a hypothetical action of $\mathbb Z_p$ by homeomorphisms on a three-manifold, it will be essential to study certain properties of orbit sets (for example, their homology).  However, we do not have the luxury of assuming such sets are at all well behaved; the most we can hope for is that we will be able to construct $\mathbb Z_p$-invariant sets which are either open or closed.  Nevertheless, \v Cech cohomology is still a reasonable object in such situations.  The purpose of this section is to develop an elementary theory centered on Alexander duality for arbitrary open and closed subsets of a manifold.  We use these tools in the proof of Theorem \ref{hs3}, specifically in Sections \ref{constructZsec}--\ref{cohomologyZsec} to construct the set $Z$ and to study its properties.

In this (and all other) sections, we always take homology and cohomology with integer coefficients.  We denote by $H_\ast$ and $H^\ast$ singular homology and cohomology, and we let $\check H^\ast$ denote \v Cech cohomology.

\subsection{\v Cech cohomology}

\begin{lemma}\label{continuityaxiom}
For a compact subset $X$ of a manifold, the natural map below is an isomorphism:
\begin{equation}\label{cechdirectlimit}
\lim_{\begin{smallmatrix}U\supseteq X\cr U\text{ open}\end{smallmatrix}}H^\ast(U)\xrightarrow\sim\check H^\ast(X)
\end{equation}
\end{lemma}

\begin{proof}
This is due to Steenrod \cite{steenrod}; see also Spanier \cite[p419]{spanier}, the key fact being that \v Cech cohomology satisfies the ``continuity axiom''.
\end{proof}

\begin{lemma}\label{mayervietoris}
If $X$ and $Y$ are two compact subsets of a manifold, then there is a (Mayer--Vietoris) long exact sequence:
\begin{equation}\label{mveq}
\cdots\to\check H^\ast(X\cup Y)\to\check H^\ast(X)\oplus\check H^\ast(Y)\to\check H^\ast(X\cap Y)\to\check H^{\ast+1}(X\cup Y)\to\cdots
\end{equation}
\end{lemma}

\begin{proof}
For arbitrary open neighborhoods $U\supseteq X$ and $V\supseteq Y$, we have a Mayer--Vietoris sequence of singular cohomology for $U$ and $V$.  Taking the direct limit over all $U$ and $V$ gives a sequence of the form (\ref{mveq}) by Lemma \ref{continuityaxiom}, and it is exact since the direct limit is an exact functor.
\end{proof}

\subsection{Alexander duality}

Alexander duality for subsets of $\mathbb S^n$ is most naturally stated using reduced homology and cohomology, which we denote by $\tilde H$.

\begin{lemma}\label{alexanderduality}
Let $X\subseteq\mathbb S^n$ be a compact set.  Then:
\begin{equation}\label{alexdualeq}
\tilde{\check H}^\ast(X)=\tilde H_{n-1-\ast}(\mathbb S^n\setminus X)
\end{equation}
\end{lemma}

\begin{proof}
If $U\subseteq\mathbb S^n$ is an open subset with smooth boundary, then it is well-known that there is a natural isomorphism $\tilde H^\ast(U)\xrightarrow\sim\tilde H_{n-1-\ast}(\mathbb S^n\setminus U)$.  Furthermore, for $U\supseteq V$ the following diagram commutes:
\begin{equation}
\begin{CD}
\tilde H^\ast(U)@>\sim >>\tilde H_{n-1-\ast}(\mathbb S^n\setminus U)\cr
@VVV@VVV\cr
\tilde H^\ast(V)@>\sim >>\tilde H_{n-1-\ast}(\mathbb S^n\setminus V)
\end{CD}
\end{equation}
Now consider the direct limit over all such open sets $U\supseteq X$ to get an isomorphism of the direct limits.  This family of sets $U$ forms a final system of neighborhoods of $X$, so by Lemma \ref{continuityaxiom} the direct limit on the left is $\tilde{\check H}^\ast(X)$.  The direct limit on the right is $\tilde H_{n-1-\ast}(\mathbb S^n\setminus X)$ by elementary properties of singular homology.
\end{proof}

\subsection{The plus operation}\label{plusopsec}

If we have a bounded subset $X\subseteq\mathbb R^3$, we want to be able to consider ``$X$ union all of the bounded connected components of $\mathbb R^3\setminus X$'' (we think of this operation as a way to simplify the set $X$, which could be very wild).  The following definition makes this precise.

\begin{definition}\label{plusdeff}
For a bounded set $X\subseteq\mathbb R^3$, we define:
\begin{equation}\label{plusdefeq}
X^+=\bigcap_{\begin{smallmatrix}\text{open }U\supseteq X\cr U\text{ bounded}\cr H_2(U)=0\end{smallmatrix}}U
\end{equation}
By Lemma \ref{alexanderduality}, for bounded open sets $U\subseteq X$, we have $H_2(U)=0\iff\check H^0(\mathbb R^3\setminus U)=\mathbb Z$.  With this reformulation, it is easy to see that if $U$ and $V$ both appear in the intersection, then so does $U\cap V$.  Let us call $X\mapsto X^+$ the \emph{plus operation}.
\end{definition}

\begin{lemma}\label{plusprop}
The plus operation satisfies the following properties:\\
\indent i.\ $X\subseteq Y\implies X^+\subseteq Y^+$.\\
\indent ii.\ $X^{++}=X^+$.\\
\indent iii.\ If $X$ is closed and bounded, then $\mathbb R^3\setminus X^+$ is the unbounded component of $\mathbb R^3\setminus X$.\\
\indent iv.\ If $X$ is closed and bounded, then $X=X^+\iff\check H^2(X)=0\iff\mathbb R^3\setminus X$ is connected.
\end{lemma}

\begin{proof}
Both (i) and (ii) are immediate from the definition.

For (iii), argue as follows.  Let $V$ denote the unbounded component of $\mathbb R^3\setminus X$.  As $V$ is open and connected, it has an exhaustion by closed connected subsets $V=\bigcup_{i=1}^\infty V_i$, where $V_1\subseteq V_2\subseteq\cdots$ and $\mathbb R^3\setminus V_i$ is bounded for all $i$.  By Lemma \ref{alexanderduality}, $H_2(\mathbb R^3\setminus V_i)=0$, and thus $X^+\subseteq\bigcap_{i=1}^\infty(\mathbb R^3\setminus V_i)=\mathbb R^3\setminus V$.  It remains to show the reverse inclusion, namely that if $W$ is a bounded component of $\mathbb R^3\setminus X$, then $W\subseteq X^+$.  Suppose $U\supseteq X$ is open, bounded, and $H_2(U)=0$.  Then $\mathbb R^3\setminus U=(\mathbb R^3\setminus(U\cup W))\cup(W\setminus U)$, which is a disjoint union of closed sets.  However $\mathbb R^3\setminus U$ is connected and $\mathbb R^3\setminus(U\cup W)$ is unbounded so \emph{a fortiori} it is nonempty.  Thus we must have $W\setminus U=\varnothing$, that is $U\supseteq W$.

For (iv), argue as follows.  Lemma \ref{alexanderduality} gives $\check H^2(X)=0\iff\mathbb R^3\setminus X$ is connected.  By (iii), we have $X=X^+\iff\mathbb R^3\setminus X$ is connected.
\end{proof}

\begin{lemma}\label{plusfinal}
Suppose $X\subseteq\mathbb R^3$ is closed and bounded with $X=X^+$.  If $\{U_\alpha\}_{\alpha\in A}$ is a final collection of open sets containing $X$, then so is $\{U_\alpha^+\}_{\alpha\in A}$.
\end{lemma}

\begin{proof}
Since $\{U_\alpha\}_{\alpha\in A}$ is final, to show that $\{U_\alpha^+\}_{\alpha\in A}$ is final it suffices to show that for every $\alpha\in A$, there exists $\beta\in A$ such that $U_\beta^+\subseteq U_\alpha$.

Thus let us suppose $\alpha\in A$ is given.  As in the proof of Lemma \ref{plusprop}(iii), there exists an exhaustion by closed connected subsets $\mathbb R^3\setminus X=\bigcup_{i=1}^\infty V_i$, where $V_1\subseteq V_2\subseteq\cdots$ and $\mathbb R^3\setminus V_i$ is bounded for all $i$.  Then for sufficiently large $i$, we have $\mathbb R^3\setminus V_i\subseteq U_\alpha$.  On the other hand, for every $i$ there exists $\beta$ such that $U_\beta\subseteq\mathbb R^3\setminus V_i$.  But now $U_\beta^+\subseteq(\mathbb R^3\setminus V_i)^+=\mathbb R^3\setminus V_i\subseteq U_\alpha$ as needed.
\end{proof}

\begin{definition}
Let $X$ be a topological space.  We say $X$ \emph{has dimension zero} iff whenever $x\in U\subseteq X$ with $U$ open, there exists a clopen (closed and open) set $V\subseteq X$ with $x\in V\subseteq U$.  It is immediate that any subspace of a space of dimension zero also has dimension zero.
\end{definition}

\begin{remark}\label{orbithasdimzero}
If $M$ is a manifold with a continuous action of $\mathbb Z_p$, then every orbit is either a finite discrete set (if the stabilizer is $p^k\mathbb Z_p$) or a cantor set (if the stabilizer is trivial).  In particular, every orbit has dimension zero.  It follows that if $A\subseteq M$ is any finite set, then $\mathbb Z_pA$ has dimension zero, as does any subset thereof.
\end{remark}

\begin{lemma}\label{nodisconnect}
If $M$ is a manifold of dimension $n\geq 2$ and $X\subseteq M$ is closed and has dimension zero, then the map $H_0(M\setminus X)\to H_0(M)$ is an isomorphism.
\end{lemma}

\begin{proof}
First, let us deal with the case $M=\mathbb S^n$ (thus $X$ is compact).  Pick a metric on $\mathbb S^n$ and fix $\epsilon>0$.  Since $X$ has dimension zero and is compact, we can cover $X$ with finitely many clopen sets $U_i\subseteq X$ ($1\leq i\leq N$) of diameter $\leq\epsilon$.  Now $\{U_i\setminus\bigcup_{1\leq j<i}U_j\}_{i=1}^N$ is a cover of $X$ by \emph{disjoint} open sets of diameter $\leq\epsilon$.  Since $\epsilon>0$ was arbitrary, this implies that $\check H^i(X)=0$ for $i>0$.  Thus by Lemma \ref{alexanderduality}, $\tilde H_0(\mathbb S^n\setminus X)=\tilde{\check H}^{n-1}(X)=0$ (since $n\geq 2$), which is sufficient.

Second, let us deal with the case $M=\mathbb R^n$.  Define $\bar X=X\cup\{\infty\}\subseteq\mathbb R^n\cup\{\infty\}=\mathbb S^n$.  We claim that $\bar X$ has dimension zero; to see this, it suffices to produce small clopen neighborhoods of $\infty\in\bar X$.  Fix $R<\infty$.  Since $X$ has dimension zero and $X\cap\overline{B(\mathbf 0,R)}$ is compact, there exist finitely many clopen sets $U_i\subseteq X$ ($1\leq i\leq N$) of diameter $\leq 1$ with $X\cap\overline{B(\mathbf 0,R)}\subseteq\bigcup_{i=1}^NU_i$.  Then $\bar X\setminus\bigcup_{i=1}^NU_i$ is clopen in $\bar X$ and contains $\infty$.  Since $R<\infty$ was arbitrary, we get arbitrarily small clopen neighborhoods of $\infty\in\bar X$.  Hence $\bar X$ has dimension zero, so by the case $M=\mathbb S^n$ dealt with above, we have $H_0(\mathbb R^n\setminus X)=H_0(\mathbb S^n\setminus\bar X)=\mathbb Z$, which is sufficient.

Now let us deal with the case of general $M$.  It suffices to show that $H^0(M)\to H^0(M\setminus X)$ is an isomorphism (recall that $H^0$ is just the group of locally constant maps to $\mathbb Z$).  Now consider any $p\in M$, and pick an open neighborhood $p\in U\cong\mathbb R^n$.  Since $X$ has dimension zero, so does $U\cap X$.  Thus by the case $M=\mathbb R^n$ dealt with above, $U\setminus X$ is connected and nonempty.  Thus any locally constant function $M\setminus X\to\mathbb Z$ can be extended uniquely to a locally constant function $M\to\mathbb Z$, as needed.
\end{proof}

\section{Nonexistence of faithful actions of $\mathbb Z_p$ on three-manifolds}

\begin{proof}[Proof of Theorem \ref{hs3}]

Let $M$ be a connected three-manifold and suppose there is a continuous injection of topological groups $\mathbb Z_p\to\operatorname{Homeo}(M)$.

\subsection{Step 1: Reduction to a local problem in $\mathbb R^3$}\label{localsec}

Let $B(r)$ denote the open ball of radius $r$ centered at the origin $\mathbf 0$ in $\mathbb R^3$.  Fix a small positive number $\eta=2^{-10}$.  In this section, we reduce to the case where $M$ is an open subset of $\mathbb R^3$ such that:\\
\indent i.\ $B(1)\subseteq M\subseteq B(1+\eta)$.\\
\indent ii.\ $d_{\mathbb R^3}(x,\alpha x)\leq\eta$ for all $x\in M$ and $\alpha\in\mathbb Z_p$.\\
\indent iii.\ No subgroup $p^k\mathbb Z_p$ fixes an open neighborhood of $\mathbf 0$.\\
This reduction is valid in any dimension.

\begin{definition}
We say that an action of a group $G$ on a topological space $X$ is \emph{locally of finite order} at a point $x\in X$ if and only if some subgroup of finite index $G'\leq G$ fixes an open neighborhood of $x$.  The action is \emph{(globally) of finite order} if and only if some subgroup of finite index $G'\leq G$ fixes all of $X$.
\end{definition}

Newman's theorem \cite[p6 Theorem 2]{newman} (see also Dress \cite[p204 Theorem 1]{dress} or Smith \cite{smith}) implies that on a connected manifold, a group action which is (everywhere) locally of finite order is globally of finite order.  Thus if $\mathbb Z_p\to\operatorname{Homeo}(M)$ is injective, then there exists a point $m\in M$ where the action is not locally of finite order.  In other words, no subgroup $p^k\mathbb Z_p$ fixes an open neighborhood of $m$.

Now fix some homeomorphism between $B(3)\subseteq\mathbb R^3$ and an open set in $M$ containing $m$, such that $\mathbf 0$ is identified with $m$.  Note that (by continuity of the map $\mathbb Z_p\to\operatorname{Homeo}(M)$) for sufficiently large $k$, we have $p^k\mathbb Z_pB(2)\subseteq B(3)$ and $d_{\mathbb R^3}(x,\alpha x)\leq\eta$ for all $x\in B(2)$ and $\alpha\in p^k\mathbb Z_p$.  Thus the action of $p^k\mathbb Z_p$ (which is isomorphic to $\mathbb Z_p$) on $M':=p^k\mathbb Z_pB(1)$ satisfies (i), (ii), and (iii) above.  Thus it suffices to replace $(M,\mathbb Z_p)$ with $(M',p^k\mathbb Z_p)$ and derive a contradiction.

Henceforth we assume that $M$ is an open subset of $\mathbb R^3$ satisfying (i), (ii), and (iii) above.

\subsection{Step 2: An invariant metric and the plus operation in $M$}\label{invmetricsec}

Equip $M$ with the following $\mathbb Z_p$-invariant metric, which induces the same topology as $d_{\mathbb R^3}$:
\begin{equation}
d_{\operatorname{inv}}(x,y):=\int_{\mathbb Z_p}d_{\mathbb R^3}(\alpha x,\alpha y)\,d\mu_{\operatorname{Harr}}(\alpha)
\end{equation}
We will never mention $d_{\operatorname{inv}}$ again explicitly, but we will often write $N^{\operatorname{inv}}_\epsilon$ for the open $\epsilon$-neighborhood in $M$ with respect to $d_{\operatorname{inv}}$.  Note that if $X\subseteq M$ is $\mathbb Z_p$-invariant, then so is $N^{\operatorname{inv}}_\epsilon(X)$.

\begin{lemma}\label{plusinvariant}
If $X\subseteq B(1)\subseteq M$, then $X^+\subseteq B(1)\subseteq M$.  If in addition $X$ is $\mathbb Z_p$-invariant, then so is $X^+$.
\end{lemma}

\begin{proof}
The first statement is clear since $B(1)$ appears in the intersection (\ref{plusdefeq}) defining $X^+$.

Now suppose that in addition $X$ is $\mathbb Z_p$-invariant.  As remarked in Definition \ref{plusdeff}, the collection of open sets appearing in the intersection (\ref{plusdefeq}) is closed under finite intersection.  In particular, we could consider just those $U$ contained in $B(1)$.  Thus \emph{a fortiori} $X^+$ is the intersection of all open sets $U$ with $X\subseteq U\subseteq M$ and $H_2(U)=0$.  The set of such $U$ is permuted by $\mathbb Z_p$, so $X^+$ is $\mathbb Z_p$-invariant.
\end{proof}

\subsection{Step 3: Construction of an interesting compact set $Z\subseteq M$}\label{constructZsec}

In this section, we construct a compact $\mathbb Z_p$-invariant set $Z\subseteq M$ such that:\\
\indent 1.\ On a coarse scale, $Z$ looks like a handlebody of genus two.\\
\indent 2.\ The action of $\mathbb Z_p$ on $\check H^1(Z)$ is nontrivial.\\
We follow a natural strategy to construct the set $Z$.  We first let $K$ be the orbit under $\mathbb Z_p$ of a closed handlebody of genus two, and we then construct a set $L$ as the orbit under $\mathbb Z_p$ of a small arc connecting two points on the boundary of $K$.  Then we let $Z=K\cup L$.  This strategy is complicated by the fact that the orbit set $K$ is \emph{a priori} very wild, and we will need to investigate the connectedness properties of it and other sets in order to make the construction work.  We will take advantage of the tools we have developed in Section \ref{plusopsec} and the setup from Sections \ref{localsec}--\ref{invmetricsec}.  Even with this preparation, though, careful verification of each step takes us a while.

\begin{figure}
\centering
\includegraphics{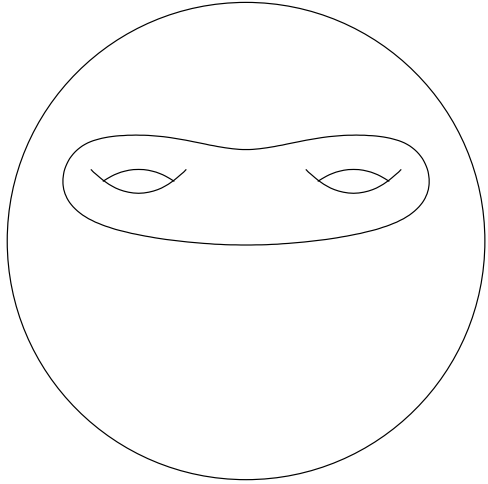}
\caption{Diagram of $K_0$ inside $B(1)\subseteq M$.}\label{diagram}
\end{figure}

Let $x_0\in B(\eta)\setminus\operatorname{Fix}\mathbb Z_p$ (such an $x_0$ exists by Step 1(iii)).

\begin{definition}[see Figure \ref{diagram}]\label{Kdef}
Let $K_0\subseteq B(1)$ be a closed handlebody of genus two whose unique point of lowest $z$-coordinate is $x_0$.  Note that the closed $\eta$-neighborhood (in the Euclidean metric) of $K_0$ is a larger handlebody isotopic to $K_0$.  Let $K=(\mathbb Z_pK_0)^+$.
\end{definition}

We think of $K$ as being illustrated essentially by Figure \ref{diagram}; it's just $K_0$ plus some wild fuzz of width $\eta$.  Note that $K$ is compact, $\mathbb Z_p$-invariant (by Lemma \ref{plusinvariant}), and $K=K^+$.  Also, $K$ is connected; actually, we will need the following stronger fact.

\begin{lemma}\label{Kpathconnected}
If $A\subseteq\partial K$ is any finite set, then $K\setminus\mathbb Z_pA$ is path-connected.
\end{lemma}

\begin{proof}
By definition, $K_0^\circ$ is path-connected.  Now the action of $\mathbb Z_p$ is within $\eta=2^{-10}$ of the identity, so $\mathbb Z_p(K_0^\circ)$ is also path-connected.

Now suppose $x\in\mathbb Z_pK_0\setminus\mathbb Z_pA$.  Then there exists $\alpha\in\mathbb Z_p$ so that $\alpha^{-1}x\in K_0$.  From $\alpha^{-1}x$ there is clearly a path inwards to $K_0^\circ$, and this is disjoint from $\mathbb Z_pA$ since $\mathbb Z_pA\subseteq\partial K$.  Translating this path by $\alpha$ shows that $\mathbb Z_pK_0\setminus\mathbb Z_pA$ is path-connected.

Now suppose $x\in K\setminus\mathbb Z_pA$.  If $x\notin\mathbb Z_pK_0$, then let $V$ denote the (open) connected component of $K\setminus\mathbb Z_pK_0$ containing $x$.  By Remark \ref{orbithasdimzero} and Lemma \ref{nodisconnect}, we can find a path from $x$ to infinity in $\mathbb R^3\setminus\mathbb Z_pA$.  There is a time when this path first hits $\partial V\subseteq\mathbb Z_pK_0$, thus giving us a path from $x$ to $\mathbb Z_pK_0$ contained in $K\setminus\mathbb Z_pA$.  Thus $K\setminus\mathbb Z_pA$ is path-connected, as was to be shown.
\end{proof}

\begin{lemma}\label{constructL}
There exist points $x_1,x_2\in K$ with $x_1\notin\operatorname{Fix}\mathbb Z_p$ and $\mathbb Z_px_1\cap\mathbb Z_px_2=\varnothing$, along with a compact $\mathbb Z_p$-invariant set $L\subseteq M$ of diameter $\leq 4\eta$ so that $L=L^+$, $K\cap L=\mathbb Z_px_1\cup\mathbb Z_px_2$, and $L$ contains a path from $x_1$ to $x_2$.
\end{lemma}

\begin{proof}
Let $x_1\in K$ be any point of lowest $z$-coordinate in $K$.  We claim that $x_1\notin\operatorname{Fix}\mathbb Z_p$.  First, observe that since $x_1\in K=(\mathbb Z_pK_0)^+$ is a point of lowest $z$-coordinate, necessarily $x_1\in\mathbb Z_pK_0$.  If $x_1\in K_0$, then $x_1=x_0$ (recall from Definition \ref{Kdef} that $x_0$ is the unique point of lowest $z$-coordinate in $K_0$), and by definition $x_0\notin\operatorname{Fix}\mathbb Z_p$.  On the other hand, if $x_1\in(\mathbb Z_pK_0)\setminus K_0$, then certainly $x_1\notin\operatorname{Fix}\mathbb Z_p$.  Thus $x_1\notin\operatorname{Fix}\mathbb Z_p$.

Since $x_1\in\mathbb Z_pK_0$, we may pick $x_2'\in\mathbb Z_p(K_0^\circ)\subseteq K^\circ$ which is arbitrarily close to $x_1$.  Specifically, let us fix $x_2'\in K^\circ\cap B(x_1,\eta)$.  Note that $x_2'\notin\mathbb Z_px_1$ since $x_1\notin K^\circ$.  Now we claim that there exists a continuous path $\gamma:[0,1]\to B(x_1,\eta)$ such that:\\
\indent i.\ $\gamma(0)=x_1$.\\
\indent ii.\ $\gamma(1)=x_2'$.\\
\indent iii.\ $\gamma^{-1}(\mathbb Z_px_1)=\{0\}$.\\
\indent iv.\ $0$ is not a limit point of $\gamma^{-1}(K)$.\\
To construct such a path $\gamma$, argue as follows.  First, define $\gamma:[0,\frac 12]\to B(x_1,\eta)$ to be some path straight downward from $x_1$ (this takes care of (i)).  Since $x_1$ is a point of smallest $z$-coordinate in $K\supseteq\mathbb Z_px_1$, this also takes care of (iv) and is consistent with (iii).  Now by Remark \ref{orbithasdimzero} and Lemma \ref{nodisconnect}, we know $B(x_1,\eta)\setminus\mathbb Z_px_1$ is path-connected, so we can find a path $\gamma:[\frac 12,1]\to B(x_1,\eta)\setminus\mathbb Z_px_1$ from $\gamma(\frac 12)$ to $x_2'$ (this takes care of (ii) and (iii)).  Splicing these two functions together gives $\gamma:[0,1]\to B(x_1,\eta)$ satisfying (i), (ii), (iii), and (iv).

Now by (ii) and (iv) we have that $t:=\min(\gamma^{-1}(K)\setminus\{0\})$ exists and is positive.  Let $x_2=\gamma(t)$, and define $L_0=\gamma([0,t])$.  By (iii), we have $\mathbb Z_px_1\cap\mathbb Z_px_2=\varnothing$.  Define $L=(\mathbb Z_pL_0)^+$, which is $\mathbb Z_p$-invariant by Lemma \ref{plusinvariant}.  By construction, $L$ contains a path from $x_1$ to $x_2$.  Certainly $L_0\subseteq B(x_1,\eta)$ so $L\subseteq B(x_1,2\eta)$ by Step 1(ii), and so $L$ has diameter $\leq 4\eta$.  It remains only to show that $K\cap L=\mathbb Z_px_1\cup\mathbb Z_px_2$ (certainly the containment $\supseteq$ is given by definition).

Note that $K\setminus\mathbb Z_pL_0=K\setminus(\mathbb Z_px_1\cup\mathbb Z_px_2)$, which by Lemma \ref{Kpathconnected} is path-connected.  Thus $K\setminus\mathbb Z_pL_0$ lies in a single connected component of $\mathbb R^3\setminus\mathbb Z_pL_0$.  Now $\mathbb Z_pL_0$ has diameter $\leq 4\eta$ as remarked above, and $K\setminus\mathbb Z_pL_0$ contains a large handlebody, so it must lie in the unbounded component of $\mathbb R^3\setminus\mathbb Z_pL_0$.  Hence $K\setminus(\mathbb Z_px_1\cup\mathbb Z_px_2)$ is disjoint from $(\mathbb Z_pL_0)^+=L$.
\end{proof}

\begin{definition}
Fix $x_1$, $x_2$, and $L$ as exist by Lemma \ref{constructL}, and define $Z=K\cup L$.
\end{definition}

\subsection{Step 4: The cohomology of $Z$}\label{cohomologyZsec}

\begin{lemma}\label{KLcohomology}
$Z=Z^+$, and the action of $\mathbb Z_p$ on $\check H^1(Z)$ is nontrivial.
\end{lemma}

\begin{proof}
The key to proving both statements is the following Mayer--Vietoris sequence (which exists and is exact by Lemma \ref{mayervietoris}):
\begin{equation}
\cdots\to\check H^\ast(Z)\to\check H^\ast(K)\oplus\check H^\ast(L)\to\check H^\ast(K\cap L)\to\check H^{\ast+1}(Z)\to\cdots
\end{equation}
Note that $\mathbb Z_p$ acts on all of these groups and that the maps are $\mathbb Z_p$-equivariant.

By Lemma \ref{constructL}, $K\cap L=\mathbb Z_px_1\cup\mathbb Z_px_2$, so by Remark \ref{orbithasdimzero}, $\check H^i(K\cap L)=0$ for $i>0$.  Thus we have $\check H^2(Z)=\check H^2(K)\oplus\check H^2(L)$.  Now applying Lemma \ref{plusprop}(iv) twice shows that $K=K^+$ and $L=L^+$ together imply that $Z=Z^+$.

Now let us show that $\mathbb Z_p$ acts on $\check H^1(Z)$ nontrivially.  We use the exact sequence:
\begin{equation}
\check H^0(K)\oplus\check H^0(L)\to\check H^0(K\cap L)\to\check H^1(Z)
\end{equation}
Note that $\check H^0$ is just the group of locally constant functions to $\mathbb Z$.  Thus it suffices to exhibit a locally constant function $q:K\cap L\to\mathbb Z$ and an element $\alpha\in\mathbb Z_p$ such that $\alpha q-q$ is not in the image of $\check H^0(K)\oplus\check H^0(L)$.  Let us define:
\begin{equation}
q(r)=\begin{cases}1&r\in p\mathbb Z_px_1\cr 0&r\in\bigcup_{a\in(\mathbb Z/p)\setminus\{0\}}(a+p\mathbb Z_p)x_1\cup\mathbb Z_px_2\end{cases}
\end{equation}
(observe by Lemma \ref{constructL} that $x_1\notin\operatorname{Fix}\mathbb Z_p$, so the sets on the right hand side are indeed disjoint).  Let $\alpha\in\mathbb Z_p$ be congruent to $1$ mod $p$.  Then we have:
\begin{equation}
(\alpha q-q)(r)=\begin{cases}-1&r\in p\mathbb Z_px_1\cr 1&r\in(1+p\mathbb Z_p)x_1\cr 0&r\in\bigcup_{a\in(\mathbb Z/p)\setminus\{0,1\}}(a+p\mathbb Z_p)x_1\cup\mathbb Z_px_2\end{cases}
\end{equation}
Now by Lemma \ref{constructL}, $x_1,x_2$ are in the same component of $L$, and by Lemma \ref{Kpathconnected} (taking $A=\varnothing$), they are in the same component of $K$.  Thus every function in the image of $\check H^0(L)\oplus\check H^0(K)$ assigns the same value to $x_1$ and $x_2$.  Clearly this is not the case for $\alpha q-q$, so we are done.
\end{proof}

We just proved that $Z=Z^+$, so by Lemma \ref{plusfinal}, $N^{\operatorname{inv}}_\epsilon(Z)^+$ is a final system of neighborhoods of $Z$.  Thus by Lemma \ref{continuityaxiom} we have $\check H^1(Z)=\varinjlim H^1(N^{\operatorname{inv}}_\epsilon(Z)^+)$ as abelian groups with an action of $\mathbb Z_p$ ($\mathbb Z_p$ acts on the latter since $N^{\operatorname{inv}}_\epsilon(Z)^+$ is $\mathbb Z_p$-invariant by Lemma \ref{plusinvariant}).  Since the action on the limit group is nontrivial, the following definition makes sense.

\begin{definition}\label{Udef}
Fix $\epsilon\in(0,\eta)$ such that the $\mathbb Z_p$ action on the image of the map $H^1(N^{\operatorname{inv}}_\epsilon(Z)^+)\to\check H^1(Z)$ is nontrivial.  Denote by $(N^{\operatorname{inv}}_\epsilon(Z)^+)_0$ the connected component containing $Z$, and define $U=(N^{\operatorname{inv}}_\epsilon(Z)^+)_0\setminus Z$, which is $\mathbb Z_p$-invariant by Lemma \ref{plusinvariant}.
\end{definition}

\subsection{Step 5: Special elements of $\mathcal S(U)$}

\begin{lemma}
The set $U$ is a quasicylinder in the sense of Definition \ref{approxsurf}.
\end{lemma}

\begin{proof}
We have $\mathbb R^3\setminus U=Z\cup(\mathbb R^3\setminus(N^{\operatorname{inv}}_\epsilon(Z)^+)_0)$; the latter two sets are connected and disjoint, so by Lemma \ref{alexanderduality}, we have $H_2(U)=\mathbb Z$.

Suppose $F\subseteq U$ is a PL embedded sphere; then $\operatorname{int}(F)$ (the bounded component of $\mathbb R^3\setminus F$) is a ball \cite{alexander}.  If $Z\subseteq\operatorname{int}(F)$, then we have a factorization $H^1(N^{\operatorname{inv}}_\epsilon(Z)^+)\to H^1(\operatorname{int}(F))\to\check H^1(Z)$, which implies the composition is trivial, contradicting the definition of $U$.  Thus $Z\nsubseteq\operatorname{int}(F)$, so $\operatorname{int}(F)\subseteq U$, that is $F$ bounds ball in $U$.  Thus $U$ is irreducible.

Certainly $U$ has at least two ends, and since $H_2(U)=\mathbb Z$ it has at most two ends.

Since $U$ is an open subset of $\mathbb R^3$, it is orientable.
\end{proof}

Recall the definition of $(\mathcal S(U),\leq)$ from Section \ref{incompressiblesection}.  Certainly the action of $\mathbb Z_p$ on $U$ induces an action on $\mathcal S_{\operatorname{TOP}}(U)$.

\begin{lemma}\label{fixedsurface}
There exists an element $\mathfrak F\in\mathcal S(U)$ which is fixed by $\mathbb Z_p$.
\end{lemma}

\begin{proof}
Observe that if $F$ is any surface in $U$, then for sufficiently large $k$, we have that $\alpha F$ is homotopic to $F$ for all $\alpha\in p^k\mathbb Z_p$.  Thus $\mathbb Z_p$ acts on $\mathcal S(U)$ with finite orbits.  Now any group acting with finite orbits on a nonempty lattice has a fixed point, namely the least upper bound of any orbit.  Recall that $\mathcal S(U)$ is nonempty by Remark \ref{nonempty}.
\end{proof}

\begin{definition}
Fix a PL $\pi_1$-injective surface $F\subseteq U$ such that $[F]\in\mathcal S(U)$ is fixed by $\mathbb Z_p$.  Denote by $\operatorname{int}(F)$ and $\operatorname{ext}(F)$ the two connected components of $\mathbb R^3\setminus F$.
\end{definition}

\begin{lemma}\label{coherentactions}
The $\mathbb Z_p$ action on $U$ induces a homomorphism $\mathbb Z_p\to\operatorname{MCG}(F)$, as well as actions on all the homology and cohomology groups appearing in (\ref{h1action})--(\ref{ccc}).  These actions are compatible with the maps in (\ref{h1action})--(\ref{ccc}), as well as with the map $\operatorname{MCG}(F)\to\operatorname{Aut}(H_1(F))$.
\begin{equation}\label{h1action}
\begin{CD}
H^1(N^{\operatorname{inv}}_\epsilon(Z)^+)@>>>H^1(\operatorname{int}(F))@>>>\check H^1(Z)\cr
@.@VVV\cr
@.H^1(F)
\end{CD}
\end{equation}
\begin{align}
\label{aaa}H_1(Z)&\to H_1(\operatorname{int}(F))\\
\label{bbb}H_1(M\setminus N^{\operatorname{inv}}_\epsilon(Z)^+)&\to H_1(\operatorname{ext}(F))\\
\label{ccc}H_1(F)&\xrightarrow\sim H_1(\operatorname{int}(F))\oplus H_1(\operatorname{ext}(F))
\end{align}
\end{lemma}

\begin{proof}
For every $\alpha\in\mathbb Z_p$, we know by Lemma \ref{homotopicisotopic} that there is a compactly supported isotopy of $U$ sending $\alpha F$ to $F$.  Now such an isotopy can clearly be extended to all of $M$ as the constant isotopy on $M\setminus U$.  Thus for every $\alpha\in\mathbb Z_p$, let us pick an isotopy:
\begin{equation}
\psi_\alpha^t:M\to M\qquad(t\in[0,1])
\end{equation}
so that $\psi_\alpha^0=\operatorname{id}_M$, $\psi_\alpha^1(\alpha F)=F$, and $\psi_\alpha^t(x)=x$ for $x\in M\setminus U$.  Denote by $T_\alpha:M\to M$ the action of $\alpha\in\mathbb Z_p$.

Now observe that for every $\alpha\in\mathbb Z_p$, the map $\psi_\alpha^1\circ T_\alpha$ fixes $Z$, $N^{\operatorname{inv}}_\epsilon(Z)^+$, and $F$.  Thus the map $\psi_\alpha^1\circ T_\alpha$ induces an automorphism of each of the diagrams (\ref{h1action})--(\ref{ccc}) and gives a compatible element of $\operatorname{MCG}(F)$.  It remains only to show that this is a homomorphism from $\mathbb Z_p$.

We need to show that for all $\alpha,\beta\in\mathbb Z_p$, the two maps $\psi_\beta^1\circ T_\beta\circ\psi_\alpha^1\circ T_\alpha$ and $\psi_{\beta\alpha}^1\circ T_{\beta\alpha}$ induce the same action on (\ref{h1action})--(\ref{ccc}) and give the same element of $\operatorname{MCG}(F)$.  It of course suffices to show that $\psi_\beta^1\circ T_\beta\circ\psi_\alpha^1\circ T_\alpha\circ(\psi_{\beta\alpha}^1\circ T_{\beta\alpha})^{-1}$ induces the trivial action on (\ref{h1action})--(\ref{ccc}) and gives the trivial element of $\operatorname{MCG}(F)$.  Now we write:
\begin{align}\label{needstobetrivial}
\psi_\beta^1\circ T_\beta\circ\psi_\alpha^1\circ T_\alpha\circ(\psi_{\beta\alpha}^1\circ T_{\beta\alpha})^{-1}
&=\psi_\beta^1\circ T_\beta\circ\psi_\alpha^1\circ T_\alpha\circ T_{\beta\alpha}^{-1}\circ(\psi_{\beta\alpha}^1)^{-1}\cr
&=\psi_\beta^1\circ T_\beta\circ\psi_\alpha^1\circ T_\beta^{-1}\circ(\psi_{\beta\alpha}^1)^{-1}
\end{align}
This is the identity map on $Z$ and on $M\setminus N^{\operatorname{inv}}_\epsilon(Z)^+$, so the action on their (co)homology is trivial.  The map (\ref{needstobetrivial}) is isotopic to the identity map via $\psi_\beta^t\circ T_\beta\circ\psi_\alpha^t\circ T_\beta^{-1}\circ(\psi_{\beta\alpha}^t)^{-1}$ for $t\in[0,1]$, which only moves points in a compact subset of $U$.  Thus the action on the cohomology of $N^{\operatorname{inv}}_\epsilon(Z)^+$ is trivial as well.  By Lemma \ref{uniquemcg}, this isotopy induces the trivial element in $\operatorname{MCG}(F)$, and this means that the action on the (co)homology of $F$, $\operatorname{int}(F)$, and $\operatorname{ext}(F)$ is trivial as well.
\end{proof}

\begin{lemma}\label{mcgimage}
The map $\mathbb Z_p\to\operatorname{MCG}(F)$ annihilates an open subgroup of $\mathbb Z_p$ and has nontrivial image.
\end{lemma}

\begin{proof}
The action of $\mathbb Z_p$ on $U$ is continuous, so for sufficiently large $k$, we have that $\alpha F$ and $F$ are homotopic as maps $F\to U$ for all $\alpha\in p^k\mathbb Z_p$.  Thus the homomorphism $\mathbb Z_p\to\operatorname{MCG}(F)$ annihilates a neighborhood of the identity in $\mathbb Z_p$.

To prove that the image is nontrivial, it suffices to show that the $\mathbb Z_p$ action on $H^1(F)$ is nontrivial.  For this, consider equation (\ref{h1action}).  By Definition \ref{Udef}, there exists an element of $H^1(N^{\operatorname{inv}}_\epsilon(Z)^+)$ whose image in $\check H^1(Z)$ is not fixed by $\mathbb Z_p$.  Thus the action of $\mathbb Z_p$ on $H^1(\operatorname{int}(F))$ is nontrivial.  Now the vertical map $H^1(\operatorname{int}(F))\to H^1(F)$ is injective (otherwise the Mayer--Vietoris sequence applied to $\mathbb R^3=\operatorname{int}(F)\cup_F\operatorname{ext}(F)$ would imply $H^1(\mathbb R^3)\ne 0$).  Thus the action of $\mathbb Z_p$ on $H^1(F)$ is nontrivial as well.
\end{proof}

\begin{lemma}\label{fixedhaslargeform}
There is a rank four submodule of $H_1(F)^{\mathbb Z_p}$ on which the intersection form is:
\begin{equation}\label{sympform}
\left(\begin{matrix}0&1\cr-1&0\end{matrix}\right)\oplus\left(\begin{matrix}0&1\cr-1&0\end{matrix}\right)
\end{equation}
\end{lemma}

\begin{proof}
There are two obvious loops in $K_0$ (Definition \ref{Kdef}) generating its homology, and two obvious dual loops in $M\setminus N^{\operatorname{inv}}_\epsilon(Z)^+$.  Since the $\mathbb Z_p$ action is within $\eta=2^{-10}$ of the identity, it is easy to see that their classes in homology are fixed by $\mathbb Z_p$.  Thus we have well-defined classes $\alpha_1,\alpha_2\in H_1(Z)^{\mathbb Z_p}$ and $\beta_1,\beta_2\in H_1(M\setminus N^{\operatorname{inv}}_\epsilon(Z)^+)^{\mathbb Z_p}$ with linking numbers $\operatorname{lk}(\alpha_i,\beta_j)=\delta_{ij}$.  Now the maps from equations (\ref{aaa}), (\ref{bbb}), and the isomorphism (\ref{ccc}) give us corresponding elements $\alpha_1,\alpha_2,\beta_1,\beta_2\in H_1(F)^{\mathbb Z_p}$.  The intersection form on $H_1(F)$ coincides with the linking pairing between $H_1(\operatorname{int}(F))$ and $H_1(\operatorname{ext}(F))$.  Thus the intersection form on the submodule of $H_1(F)^{\mathbb Z_p}$ generated by $\alpha_1,\alpha_2,\beta_1,\beta_2$ is indeed given by (\ref{sympform}).
\end{proof}

By Lemma \ref{mcgimage}, the image of $\mathbb Z_p$ in $\operatorname{MCG}(F)$ is a nontrivial cyclic $p$-group.  It has a (unique) subgroup isomorphic to $\mathbb Z/p$, and by Lemma \ref{fixedhaslargeform} this subgroup $\mathbb Z/p\subseteq\operatorname{MCG}(F)$ has the property that $H_1(F)^{\mathbb Z/p}$ has a submodule on which the intersection form is given by (\ref{sympform}).  This contradicts Lemma \ref{homologyaction} below, and thus completes the proof of Theorem \ref{hs3}.
\end{proof}

\begin{remark}
One expects that if we make $U$ thick enough around most of $Z$ (but still thin near $L$), then away from $L$ the surface $F$ can be made to look like a punctured surface of genus two.  However, it is not clear how to prove this stronger, more geometric property about $F$.

If we could prove this, then the final analysis is more robust.  Instead of homology classes $\alpha_1,\alpha_2,\beta_1,\beta_2\in H_1(F)^{\mathbb Z_p}$, we now have simple closed curves $\alpha_1,\alpha_2,\beta_1,\beta_2$ in $F$ which are fixed up to isotopy by the $\mathbb Z_p$ action.  Now the finite image of $\mathbb Z_p$ in $\operatorname{MCG}(F)$ is realized as a subgroup of $\operatorname{Isom}(F,g)$ for some hyperbolic metric $g$ on $F$.  A hyperbolic isometry fixing $\alpha_1,\alpha_2,\beta_1,\beta_2$ fixes their unique length-minimizing geodesic representatives, and thus fixes their intersections, forcing the isometry to be the identity map.  This contradicts the nontriviality of the homomorphism $\mathbb Z_p\to\operatorname{MCG}(F)$.  As the reader will readily observe, it is enough that $\alpha_1,\beta_1$ be fixed, so we could let $K_0$ be a handlebody of genus one and this argument would go through.
\end{remark}

\section{Actions of $\mathbb Z/p$ on closed surfaces}

\begin{lemma}\label{homologyaction}
Let $F$ be a closed connected oriented surface, and let $\mathbb Z/p\subseteq\operatorname{MCG}(F)$ be some cyclic subgroup of prime order.  Then the intersection form of $F$ restricted to $H_1(F)^{\mathbb Z/p}$ and taken modulo $p$ has rank at most two.
\end{lemma}

The proof below follows Symonds \cite[pp389--390 Theorem B]{symonds} and Chen--Glover--Jensen \cite[Theorem 1.1]{chengloverjensen}, who make explicit a classification due to Nielsen \cite{nielsen2} (in fact, Symonds' result actually implies ours for $p>2$).

\begin{proof}
We seek simply to classify all subgroups $\mathbb Z/p\subseteq\operatorname{MCG}(F)$ and prove that in each case, the conclusion of the lemma is satisfied.  This classification is essentially due to Nielsen \cite{nielsen2}, who showed that finite cyclic subgroups $G\subseteq\operatorname{MCG}(F)$ are classifed up to conjugacy by their ``fixed point data''.

By work of Nielsen \cite{nielsen}, any $\mathbb Z/p\subseteq\operatorname{MCG}(F)$ is realized by a genuine action of $\mathbb Z/p$ on $F$ by isometries in some metric (see Thurston \cite{thurston} or Kerckhoff \cite{kerckhoff} for a modern perspective on this fact and its generalization to any finite group).   Now let us switch notation and write $\tilde S=F$, $\mathcal S=F/(\mathbb Z/p)$ (the quotient orbifold), and $S$ for the coarse space of $\mathcal S$.  The cover $\tilde S\to\mathcal S$ is classified by an element $\alpha\in H^1(\mathcal S,\mathbb Z/p)$, which is nonzero since $\tilde S$ is connected.  Let $g$ denote the genus of $S$, and $n$ the number of orbifold points of $\mathcal S$.

Now for $(\tilde S,\mathcal S,\alpha,g,n)$ as above, we prove the following more precise statement (which certainly implies the lemma):
\begin{equation}\label{fineintersectionformresult}
\text{the intersection form on }H_1(\tilde S)^{\mathbb Z/p}\cong
\begin{cases}
\left(\begin{smallmatrix}0&p\cr-p&0\end{smallmatrix}\right)^{\oplus(g-1)}\oplus\left(\begin{smallmatrix}0&1\cr-1&0\end{smallmatrix}\right)&n=0\cr
\left(\begin{smallmatrix}0&p\cr-p&0\end{smallmatrix}\right)^{\oplus g}&n>0
\end{cases}
\end{equation}

First, let us treat the case $n=0$ (so $\mathcal S=S$).  It is well-known that both compositions $\operatorname{MCG}(\Sigma_g)\to\operatorname{Sp}(2g,\mathbb Z)\to\operatorname{Sp}(2g,\mathbb Z/p)$ are surjective, and that $\operatorname{Sp}(2g,\mathbb Z/p)$ acts transitively on the nonzero elements of $(\mathbb Z/p)^{\oplus 2g}$.  Thus without loss of generality, it suffices to consider one particular nonzero $\alpha\in H^1(S,\mathbb Z/p)$.  Thus, let us suppose that $\alpha$ is the Poincar\'e dual of a nonseparating simple closed curve $\ell$ on $S$.  Then the cover $\tilde S$ is obtained by cutting $S$ along $\ell$ and gluing together in a circle $p$ copies of the resulting cut surface.  Now one easily observes that there is a direct sum decomposition respecting intersection forms $H_1(\tilde S)\cong H_1(\Sigma_1)\oplus H_1(\Sigma_{g-1})^{\oplus p}$, where $\mathbb Z/p$ acts trivially on $H_1(\Sigma_1)$ and acts by rotation of the summands on $H_1(\Sigma_{g-1})^{\oplus p}$.  Then $H_1(\tilde S)^{\mathbb Z/p}$ is $H_1(\Sigma_1)$ with its usual intersection form, plus a copy of $H_1(\Sigma_{g-1})$ with its intersection form multiplied by $p$.  Hence equation (\ref{fineintersectionformresult}) holds in the case $n=0$.

Now, let us treat the case $n>0$.  Label the orbifold points $p_1,\ldots,p_n\in\mathcal S$, and define $r_i:H^1(\mathcal S,\mathbb Z/p)\to\mathbb Z/p$ to be the evaluation on a small loop around $p_i$.  By construction, the orbifold points are precisely the ramification points of $\tilde S\to S$, and hence $r_i(\alpha)\ne 0$ for all $i$.  On the other hand, if $\mathcal D\subseteq\mathcal S$ is a disk containing all $n$ orbifold points, then $\partial\mathcal D=\partial(\mathcal S\setminus\mathcal D)$ is null-homologous in $\mathcal S$, so $\sum_{i=1}^nr_i(\alpha)=0$.  Thus in particular $n\geq 2$.

Next, let us show how to reduce to the case where either $g=0$ or $n=2$.  So, suppose $g\geq 1$ and $n\geq 3$, and let $\mathcal D\subseteq\mathcal S$ be a disk containing $p_1,\ldots,p_{n-1}$.  Take $\mathcal D$ and $\mathcal S\setminus\mathcal D$ and glue in a disk with a single $\mathbb Z/p$ orbifold point to each, and call the resulting closed orbifolds $\mathcal S'$ and $\mathcal S''$ respectively.  The class $\alpha\in H^1(\mathcal S,\mathbb Z/p)$ naturally induces $\alpha'\in H^1(\mathcal S',\mathbb Z/p)$ and $\alpha''\in H^1(\mathcal S'',\mathbb Z/p)$, and so we get covers $\tilde S'$ and $\tilde S''$.  Since $\langle\alpha,\partial\mathcal D\rangle=\sum_{i=1}^{n-1}r_i(\alpha)=-r_n(\alpha)\ne 0$, it follows that $\partial\mathcal D$ lifts to a single separating loop in $\tilde S$, which exhibits $\tilde S$ as a connected sum of $\tilde S'$ and $\tilde S''$.  Hence $H_1(\tilde S)=H_1(\tilde S')\oplus H_1(\tilde S'')$ respecting the intersection form and the $\mathbb Z/p$ action.  Now $(g(S'),n(\mathcal S'))=(0,n(\mathcal S))$ and $(g(S''),n(\mathcal S''))=(g(S),2)$, so knowing equation (\ref{fineintersectionformresult}) for $\tilde S'$ and $\tilde S''$ implies it for $\tilde S$.  Hence it suffices to deal with the two cases $g=0$ and $n=2$.

Case $g=0$.  Embed a tree $T=(V,E)$ in $S$, where $V=\{p_1,\ldots,p_{n-1}\}$ and where no edge passes through $p_n$.  This gives a cell decomposition of $S$, which we lift to a $(\mathbb Z/p)$-equivariant cell decomposition of $\tilde S$.  Now let us consider the cellular chains computing $H_1(\tilde S)$.  There is exactly one $2$-cell, and it has zero boundary.  Thus $H_1(\tilde S)$ is simply the group of $1$-cycles.  Now a $1$-chain is fixed by $\mathbb Z/p$ iff for all $e\in E$ it assigns the same weight to each of the $p$ lifts of $e$.  Now it is easy to observe that since $T$ is a tree, such a $1$-chain cannot be a $1$-cycle unless it vanishes.  Thus $H_1(\tilde S)^{\mathbb Z/p}=0$, so equation (\ref{fineintersectionformresult}) holds in the case $g=0$.

Case $n=2$.  There is a natural exact sequence:
\begin{equation}
0\to H^1(S,\mathbb Z/p)\to H^1(\mathcal S,\mathbb Z/p)\xrightarrow{(r_1,r_2)}(\mathbb Z/p)^{\oplus 2}\xrightarrow{+}\mathbb Z/p\to 0
\end{equation}
Now, multiplying $\alpha$ by a nonzero element of $\mathbb Z/p$ yields an equivalent problem, so we assume without loss of generality that $\alpha\in A:=(r_1,r_2)^{-1}(1,-1)$.  We claim that $\operatorname{MCG}(\mathcal S\text{ rel }\{p_1,p_2\})$ acts transitively on $A$.  Let $[\ell]\in A$ be the Poincar\'e dual of a simple arc $\ell$ from $p_1$ to $p_2$.  We observed earlier that $\operatorname{MCG}(S\text{ rel }\ell)$ acts transitively on the nonzero elements of $H^1(S,\mathbb Z/p)$; hence $\operatorname{MCG}(\mathcal S\text{ rel }\{p_1,p_2\})$ acts transitively on $A\setminus\{[\ell]\}$ (by exactness).  On the other hand, since $g\geq 1$, there certainly exists $\gamma\in\operatorname{MCG}(\mathcal S\text{ rel }\{p_1,p_2\})$ for which $\gamma[\ell]-[\ell]\in H^1(S,\mathbb Z/p)$ is nonzero (for instance, $\gamma$ could be a Dehn twist around a nonseparating simple closed curve which intersects $\ell$ exactly once), and so $[\ell]$ is in the same orbit as $A\setminus\{[\ell]\}$.  Thus $\operatorname{MCG}(\mathcal S\text{ rel }\{p_1,p_2\})$ acts transitively on $A$.  Hence we may assume without loss of generality that $\alpha=[\ell]$.  Now we can see the cover $\tilde S\to S$ explicitly: we cut $S$ along $\ell$ and glue together $p$ copies in the relevant fashion.  One then easily observes that there is an isomorphism respecting intersection forms $H_1(\tilde S)\cong H_1(\Sigma_g)^{\oplus p}$, where $\mathbb Z/p$ acts by rotation of the summands.  Then $H_1(\tilde S)^{\mathbb Z/p}$ is a copy of $H_1(\Sigma_g)$ with its intersection form multiplied by $p$.  Hence equation (\ref{fineintersectionformresult}) holds in the case $n=2$ as well, and the proof is complete.
\end{proof}

\bibliographystyle{plain}
\bibliography{hilbertsmith}

\newcommand{\chsort}[1]{}
\begin{thebibliography}{10}

\bibitem{mathoverflow}
Agol\phantom{x}(\texttt{mathoverflow.net/users/1345}).
\newblock Incompressible surfaces in an open subset of {R}{\char`\^}3.
\newblock MathOverflow.
\newblock \texttt{http://mathoverflow.net/questions/74935} (version:
  2011-09-07).

\bibitem{alexander}
J.~W. Alexander.
\newblock On the subdivision of 3-space by a polyhedron.
\newblock {\em Proceedings of the National Academy of Sciences of the United
  States of America}, 10(1):pp. 6--8, 1924.

\bibitem{bing}
R.~H. Bing.
\newblock An alternative proof that {$3$}-manifolds can be triangulated.
\newblock {\em Ann. of Math. (2)}, 69:37--65, 1959.

\bibitem{bochnermontgomery}
Salomon Bochner and Deane Montgomery.
\newblock Locally compact groups of differentiable transformations.
\newblock {\em Ann. of Math. (2)}, 47:639--653, 1946.

\bibitem{chengloverjensen}
Yu~Qing Chen, Henry~H. Glover, and Craig~A. Jensen.
\newblock Prime order subgroups of mapping class groups.
\newblock {\em JP J. Geom. Topol.}, 11(2):87--99, 2011.

\bibitem{dress}
Andreas Dress.
\newblock Newman's theorems on transformation groups.
\newblock {\em Topology}, 8:203--207, 1969.

\bibitem{freedman}
Michael Freedman, Joel Hass, and Peter Scott.
\newblock Least area incompressible surfaces in {$3$}-manifolds.
\newblock {\em Invent. Math.}, 71(3):609--642, 1983.

\bibitem{gleason}
A.~M. Gleason.
\newblock The structure of locally compact groups.
\newblock {\em Duke Math. J.}, 18:85--104, 1951.

\bibitem{gleason2}
Andrew~M. Gleason.
\newblock Groups without small subgroups.
\newblock {\em Ann. of Math. (2)}, 56:193--212, 1952.

\bibitem{gottschalk}
W.~H. Gottschalk.
\newblock Minimal sets: an introduction to topological dynamics.
\newblock {\em Bull. Amer. Math. Soc.}, 64:336--351, 1958.

\bibitem{gulliver}
Robert~D. Gulliver, II.
\newblock Regularity of minimizing surfaces of prescribed mean curvature.
\newblock {\em Ann. of Math. (2)}, 97:275--305, 1973.

\bibitem{hamilton}
A.~J.~S. Hamilton.
\newblock The triangulation of {$3$}-manifolds.
\newblock {\em Quart. J. Math. Oxford Ser. (2)}, 27(105):63--70, 1976.

\bibitem{jacorubinstein}
William Jaco and J.~Hyam Rubinstein.
\newblock P{L} minimal surfaces in {$3$}-manifolds.
\newblock {\em J. Differential Geom.}, 27(3):493--524, 1988.

\bibitem{kakimizu}
Osamu Kakimizu.
\newblock Finding disjoint incompressible spanning surfaces for a link.
\newblock {\em Hiroshima Math. J.}, 22(2):225--236, 1992.

\bibitem{kerckhoff}
Steven~P. Kerckhoff.
\newblock The {N}ielsen realization problem.
\newblock {\em Ann. of Math. (2)}, 117(2):235--265, 1983.

\bibitem{kirsie}
Robion~C. Kirby and Laurence~C. Siebenmann.
\newblock {\em Foundational essays on topological manifolds, smoothings, and
  triangulations}.
\newblock Princeton University Press, Princeton, N.J., 1977.
\newblock With notes by John Milnor and Michael Atiyah, Annals of Mathematics
  Studies, No. 88.

\bibitem{lee}
Joo~Sung Lee.
\newblock Totally disconnected groups, {$p$}-adic groups and the
  {H}ilbert-{S}mith conjecture.
\newblock {\em Commun. Korean Math. Soc.}, 12(3):691--699, 1997.

\bibitem{hsholder}
{\u{I}}ozhe Maleshich.
\newblock The {H}ilbert-{S}mith conjecture for {H}\"older actions.
\newblock {\em Uspekhi Mat. Nauk}, 52(2(314)):173--174, 1997.

\bibitem{hsqc}
Gaven~J. Martin.
\newblock The {H}ilbert-{S}mith conjecture for quasiconformal actions.
\newblock {\em Electron. Res. Announc. Amer. Math. Soc.}, 5:66--70
  (electronic), 1999.

\bibitem{mj}
Mahan Mj.
\newblock Pattern rigidity and the {H}ilbert--{S}mith conjecture.
\newblock {\em Geom. Topol.}, 16(2):1205--1246, 2012.

\bibitem{moise}
Edwin~E. Moise.
\newblock Affine structures in {$3$}-manifolds. {V}. {T}he triangulation
  theorem and {H}auptvermutung.
\newblock {\em Ann. of Math. (2)}, 56:96--114, 1952.

\bibitem{montgomeryzippin}
Deane Montgomery and Leo Zippin.
\newblock Small subgroups of finite-dimensional groups.
\newblock {\em Ann. of Math. (2)}, 56:213--241, 1952.

\bibitem{montgomeryzippin2}
Deane Montgomery and Leo Zippin.
\newblock {\em Topological transformation groups}.
\newblock Interscience Publishers, New York-London, 1955.

\bibitem{newman}
M.~H.~A. Newman.
\newblock A theorem on periodic transformations of spaces.
\newblock {\em Quart. J. Math.}, os-2(1):1--8, 1931.

\bibitem{nielsen2}
J.~Nielsen.
\newblock Die struktur periodischer transformationen von fl\"achen.
\newblock {\em Danske Vid. Selsk, Mat.-Fys. Medd.}, 15:1--77, 1937.

\bibitem{nielsen}
Jakob Nielsen.
\newblock Abbildungsklassen endlicher {O}rdnung.
\newblock {\em Acta Math.}, 75:23--115, 1943.

\bibitem{osserman2}
Robert Osserman.
\newblock {\em A survey of minimal surfaces}.
\newblock Van Nostrand Reinhold Co., New York, 1969.

\bibitem{osserman1}
Robert Osserman.
\newblock A proof of the regularity everywhere of the classical solution to
  {P}lateau's problem.
\newblock {\em Ann. of Math. (2)}, 91:550--569, 1970.

\bibitem{papa}
C.~D. Papakyriakopoulos.
\newblock On {D}ehn's lemma and the asphericity of knots.
\newblock {\em Ann. of Math. (2)}, 66:1--26, 1957.

\bibitem{przytyckischultens}
Piotr Przytycki and Jennifer Schultens.
\newblock Contractibility of the {K}akimizu complex and symmetric {S}eifert
  surfaces.
\newblock {\em Trans. Amer. Math. Soc.}, 364(3):1489--1508, 2012.

\bibitem{raymondwilliams}
Frank Raymond and R.~F. Williams.
\newblock Examples of {$p$}-adic transformation groups.
\newblock {\em Ann. of Math. (2)}, 78:92--106, 1963.

\bibitem{hslipschitz}
Du{\u{s}}an Repov{\u{s}} and Evgenij {\u{S}}{\u{c}}epin.
\newblock A proof of the {H}ilbert-{S}mith conjecture for actions by
  {L}ipschitz maps.
\newblock {\em Math. Ann.}, 308(2):361--364, 1997.

\bibitem{sacksuhlenbeck}
J.~Sacks and K.~Uhlenbeck.
\newblock The existence of minimal immersions of {$2$}-spheres.
\newblock {\em Ann. of Math. (2)}, 113(1):1--24, 1981.

\bibitem{sacksuhlenbeck2}
J.~Sacks and K.~Uhlenbeck.
\newblock Minimal immersions of closed {R}iemann surfaces.
\newblock {\em Trans. Amer. Math. Soc.}, 271(2):639--652, 1982.

\bibitem{schoenyau}
R.~Schoen and Shing~Tung Yau.
\newblock Existence of incompressible minimal surfaces and the topology of
  three-dimensional manifolds with nonnegative scalar curvature.
\newblock {\em Ann. of Math. (2)}, 110(1):127--142, 1979.

\bibitem{schultens}
Jennifer Schultens.
\newblock The {K}akimizu complex is simply connected.
\newblock {\em J. Topol.}, 3(4):883--900, 2010.
\newblock With an appendix by Michael Kapovich.

\bibitem{smith}
P.~A. Smith.
\newblock Transformations of finite period. {III}. {N}ewman's theorem.
\newblock {\em Ann. of Math. (2)}, 42:446--458, 1941.

\bibitem{spanier}
Edwin~H. Spanier.
\newblock Cohomology theory for general spaces.
\newblock {\em Ann. of Math. (2)}, 49:407--427, 1948.

\bibitem{steenrod}
Norman~E. Steenrod.
\newblock Universal {H}omology {G}roups.
\newblock {\em Amer. J. Math.}, 58(4):661--701, 1936.

\bibitem{symonds}
Peter Symonds.
\newblock The cohomology representation of an action of {$C_p$} on a surface.
\newblock {\em Trans. Amer. Math. Soc.}, 306(1):389--400, 1988.

\bibitem{taoblog}
Terence Tao.
\newblock Hilbert's fifth problem and related topics.
\newblock Manuscript, 2012.
\newblock
  \\\texttt{http://terrytao.wordpress.com/books/hilberts-fifth-problem-and-rel%
ated-topics/}.

\bibitem{thurston}
William~P. Thurston.
\newblock On the geometry and dynamics of diffeomorphisms of surfaces.
\newblock {\em Bull. Amer. Math. Soc. (N.S.)}, 19(2):417--431, 1988.

\bibitem{waldhausen}
Friedhelm Waldhausen.
\newblock On irreducible {$3$}-manifolds which are sufficiently large.
\newblock {\em Ann. of Math. (2)}, 87:56--88, 1968.

\bibitem{walsh}
John~J. Walsh.
\newblock Light open and open mappings on manifolds. {II}.
\newblock {\em Trans. Amer. Math. Soc.}, 217:271--284, 1976.

\bibitem{wilson}
David Wilson.
\newblock Open mappings on manifolds and a counterexample to the {W}hyburn
  conjecture.
\newblock {\em Duke Math. J.}, 40:705--716, 1973.

\bibitem{yamabe1}
Hidehiko Yamabe.
\newblock On the conjecture of {I}wasawa and {G}leason.
\newblock {\em Ann. of Math. (2)}, 58:48--54, 1953\chsort{a}.

\bibitem{yamabe2}
Hidehiko Yamabe.
\newblock A generalization of a theorem of {G}leason.
\newblock {\em Ann. of Math. (2)}, 58:351--365, 1953\chsort{b}.

\bibitem{ctyang}
Chung-Tao Yang.
\newblock {$p$}-adic transformation groups.
\newblock {\em Michigan Math. J.}, 7:201--218, 1960.

\end{thebibliography}

\end{document}